\documentclass[12pt,reqno]{amsart}
\usepackage[margin=25mm]{geometry}
\usepackage[skins]{tcolorbox}

\usepackage{blindtext}
\usepackage{graphicx}
\usepackage{amssymb}
\usepackage{amsfonts}
\usepackage{amsthm}
\usepackage{amsmath}
\usepackage{eucal}
\usepackage{pdfsync}
\usepackage{hyperref}
\usepackage{enumerate}
\usepackage{mathrsfs}
\usepackage{dsfont}
\usepackage{textcomp}
\usepackage{tipa}
\usepackage{mathtools}
\usepackage[english]{babel}
\usepackage[autostyle]{csquotes}
\usepackage{xcolor}
\usepackage{enumerate}
\usepackage{upgreek}
\usepackage{tikz}
\usetikzlibrary{matrix}
\usepackage{xspace}
\theoremstyle{plain}
\usepackage[utf8]{inputenc}
\usepackage[T1]{fontenc}
\usepackage{lmodern}

\def\Q{{\mathbb Q}}

\newtheorem{lem}{Lemma}[section]

\newtheorem{prop}[lem]{Proposition}

\newtheorem{claim*}{Claim}
\newtheorem{thm}[lem]{Theorem}

\theoremstyle{rem}
\newtheorem{rem}[lem]{Remark}

\tcolorboxenvironment{thm}{blanker,before skip=10pt,after skip=10pt}
\tcolorboxenvironment{lem}{blanker,before skip=10pt,after skip=10pt}
\tcolorboxenvironment{prop}{blanker,before skip=10pt,after skip=10pt}

\theoremstyle{remark}

\title[Growth of torsion]{NON-CYCLIC TORSION OF ELLIPTIC CURVES OVER IMAGINARY QUADRATIC FIELDS OF CLASS NUMBER 1 }

\author{IRMAK BAL\c{C}IK}  
\address{Department of Mathematics, Northwestern University, 2033 Sheridan Road,
	Evanston, IL 60208, USA}
\email{irmak.balcik@northwestern.edu}

\begin{document}

\begin{abstract}
	Let $K$ be a non-cyclotomic imaginary quadratic field with class number 1 and $E/K$ an elliptic curve with $E(K)[2]\simeq \mathbb{Z}/2\mathbb{Z} \oplus \mathbb{Z}/2\mathbb{Z}.$ In this paper, we determine the torsion groups that can arise as $E(L)_{\text{tor}}$ where $L$ is any quadratic extension of $K.$
\end{abstract}

\maketitle

\section{Introduction}
Let $E$ be an elliptic curve over a number field $K$. The celebrated theorem of Mordell-Weil asserts that the set of $K$-rational points $E(K)$ is a finitely generated abelian group. Let $E(K)_{\text{tor}}$ denote the torsion subgroup of $E(K)$. Mazur \cite{Maz78} first completed the classification of torsion subgroups that can be realized as $E(\mathbb{Q})_{\text{tor}}$ where $E$ is an elliptic curve over $\mathbb{Q}.$ These groups are:
\begin{equation} \label{eq:mazurgroups}
\begin{array}{ll}
\mathbb{Z}/n\mathbb{Z} & 1\leq n \leq 12,\  n\neq 11 \\
\mathbb{Z}/2\mathbb{Z} \oplus \mathbb{Z}/{2n}\mathbb{Z} & 1\leq n\leq 4.
\end{array}
\end{equation}

Subsequently, Kamienny \cite{Kam92} and Kenku-Momose \cite{KM88} determined the complete list of torsion subgroups arising as $E(K)_{\text{tor}}$ where $E$ is an elliptic curve over any quadratic number field $K$. These groups are:
\begin{equation} \label{eq:kamiennygroups}
\begin{array}{ll}
\mathbb{Z}/n\mathbb{Z} &  1\leq n\leq 18,\ n\neq 17  \\  \mathbb{Z}/2\mathbb{Z} \oplus \mathbb{Z}/{2n}\mathbb{Z} & 1\leq n\leq 6  \\ 
\mathbb{Z}/3\mathbb{Z} \oplus \mathbb{Z}/{3n}\mathbb{Z} & 1\leq n\leq 2 \\ 
\mathbb{Z}/4\mathbb{Z} \oplus \mathbb{Z}/4\mathbb{Z}. & 
\end{array}
\end{equation}

Over cubic number fields, the full description of torsion subgroups has been settled only recently \cite{DEVMB21}. Over quartic number fields, while we still do not have a complete classification, analogous to \eqref{eq:mazurgroups}, we know due to \cite{JKP06} that which torsion subgroups occur infinitely often up to isomorphism of elliptic curves. Moreover, $17$ is the largest prime dividing the order of a point over a quartic number field by \cite{DKSS17}. As a first step toward this direction, we ask ``given an elliptic curve $E$ over any quadratic number field $K$, if $E(K)_{\text{tor}}$ has full 2-torsion, then how does $E(L)_{\text{tor}}$ relate with $E(K)_{\text{tor}}$ where $[L:K]=2$?" If $E(K)_{\text{tor}} \subsetneq E(L)_{\text{tor}},$ we say that torsion grows. In particular, for a given elliptic curve $E$ over $K$,  we know $E(L)_{\text{tor}} = E(K)_{\text{tor}}$ for all but finitely many quadratic extensions $L/K$ since there are only finitely many torsion groups that can arise over a quartic number field. This leads us to a more fundamental question: whether there is a good bound for the growth? Kwon \cite{Kwo97} showed that for an elliptic curve $E/\mathbb{Q},$ the growth of $E(\mathbb{Q})_{\text{tor}}$ upon quadratic base change is controlled by the torsion structure realizing on a quadratic twist of the original curve. His result can be generalized to number fields which allows us to bound $E(L)_{\text{tor}}/E(K)_{\text{tor}}$ (See Proposition \ref{mainprop1} below). In this note, we answer the main objective above for the set $\mathcal{S}$ of all non-cyclotomic imaginary quadratic fields with class number 1, that is
$  \mathcal{S}=\{ \mathbb{Q}(\sqrt{-2}), \mathbb{Q}(\sqrt{-7}), \mathbb{Q}(\sqrt{-11}),\mathbb{Q}(\sqrt{-19}),\mathbb{Q}(\sqrt{-43}),\mathbb{Q}(\sqrt{-67}), \mathbb{Q}(\sqrt{-163})\}.$

We choose to work over the base fields in $\mathcal{S}$ as these are the remaining imaginary quadratic fields with class number 1 for which the classification of torsion subgroups is known due to \cite{SS18}. The method used in their paper is outlined by Kamienny and Najman \cite{KN12}. In case $E(K)[2]\simeq \mathbb{Z}/2\mathbb{Z} \oplus \mathbb{Z}/2\mathbb{Z},$ we generalize both Kwon's work in which $K=\mathbb{Q},$ and Newman's work \cite{New17} in which $K \in \{\mathbb{Q}(\sqrt{-1}),\mathbb{Q}(\sqrt{-3})\}.$ 
This note also completes the author's earlier work \cite{Bal21} in which for any $K \in \mathcal{S}$ and an elliptic curve $E/K$ with trivial $2$-torsion, the growth of $E(K)_{\text{tor}}$ is classified upon quadratic base change.

Our methods do not explicitly rely upon the facts that $\mathcal{O}_K$ is a UFD and has a finite unit group whereas Kwon's techniques do. Instead, we used new techniques. First, the action of Gal$(L/K)$ on $E(L)_{\text{tor}}$ when $L/K$ is Galois. Second, we  made Newman's method more general and efficient, specifically by avoiding the use of Gr{\"o}bner basis (see Lemma \ref{lem3}). Altogether this finishes the classification of quadratic growth of non-cyclic torsion for the set of all imaginary quadratic number fields with class number $1$.

In order to describe our main result, we introduce the following notation. As in \cite{GJT14}, let
$$ \Phi_K(d):= \{G : \exists \ E/K : E(L)_{\text{tor}} \simeq G \  \text{for some}\ [L:K]=d\} $$
for  $d \in \mathbb{Z}^+.$ For a fixed $G \in \Phi_K(1),$  we write
\begin{align*}
 \Phi_{K}(d,G):= \{ H : \exists \ E/K :  E(K)_{\text{tor}} \simeq G, E(L)_{\text{tor}} \simeq H\ \text{for some}\ [L:K]=d \}.
\end{align*}

We are ready to state our main result for which one can find explicit examples at the end of introduction.

\begin{thm}\label{mainthm2}   
	Let $K=\mathbb{Q}(\sqrt{D})$ be fixed  in $\mathcal{S}$ with $D=-2,-7,-11,-19,-43,-67,-163$ and let $G \in \Phi_K(1)$ such that $\mathbb{Z}/2\mathbb{Z} \oplus \mathbb{Z}/2\mathbb{Z} \subseteq G.$

	\begin{enumerate}
		\item If $G \simeq \mathbb{Z}/2\mathbb{Z} \oplus \mathbb{Z}/{12}\mathbb{Z}$ then $\Phi_K(2,G) =  \{\mathbb{Z}/2\mathbb{Z} \oplus \mathbb{Z}/{12}\mathbb{Z}\}.$ 
		\item If $G \simeq \mathbb{Z}/2\mathbb{Z} \oplus \mathbb{Z}/{10}\mathbb{Z},$ then
		 $\Phi_K(2,G) = \{\mathbb{Z}/2\mathbb{Z} \oplus \mathbb{Z}/{10}\mathbb{Z}\}.$ 
		\item If $G \simeq \mathbb{Z}/2\mathbb{Z} \oplus \mathbb{Z}/8\mathbb{Z}$ and
		\begin{enumerate}[i.]
			\item $D\neq -7,$ then  $\Phi_K(2,G) =\{\mathbb{Z}/2\mathbb{Z} \oplus \mathbb{Z}/{8}\mathbb{Z}\}.$
			
             \item $D=-7,$ then $\Phi_K(2,G) = \{\mathbb{Z}/2\mathbb{Z} \oplus \mathbb{Z}/{8}\mathbb{Z},  \mathbb{Z}/2\mathbb{Z} \oplus \mathbb{Z}/{16}\mathbb{Z},  \mathbb{Z}/4\mathbb{Z} \oplus \mathbb{Z}/8\mathbb{Z} \}.$
		\end{enumerate}
		\item	If $G \simeq \mathbb{Z}/2\mathbb{Z} \oplus \mathbb{Z}/6\mathbb{Z}$ and
		\begin{enumerate}[i.]
			\item $D=-2,$ then $\Phi_K(2,G) =\{\mathbb{Z}/2\mathbb{Z} \oplus \mathbb{Z}/{6}\mathbb{Z}, \mathbb{Z}/2\mathbb{Z} \oplus \mathbb{Z}/{12}\mathbb{Z}\}.$
			\item $D\neq -2,$ then 
			$\Phi_K(2,G) =\{\mathbb{Z}/2\mathbb{Z} \oplus \mathbb{Z}/{6}\mathbb{Z}, \mathbb{Z}/2\mathbb{Z} \oplus \mathbb{Z}/{12}\mathbb{Z},  \mathbb{Z}/6\mathbb{Z} \oplus \mathbb{Z}/{6}\mathbb{Z}\}.$
			\end{enumerate}
		\item   If $G \simeq \mathbb{Z}/2\mathbb{Z} \oplus \mathbb{Z}/4\mathbb{Z}$ and 
			\begin{enumerate}[i.] 
				\item $D=-2,-11,$ then $\Phi_K(2,G)=  \{\mathbb{Z}/2\mathbb{Z} \oplus \mathbb{Z}/4\mathbb{Z},  \mathbb{Z}/2\mathbb{Z} \oplus \mathbb{Z}/{8}\mathbb{Z}, \mathbb{Z}/4\mathbb{Z} \oplus \mathbb{Z}/4\mathbb{Z}\}.$ 
				\item $D=-7,$ then $\Phi_K(2,G)=\{\mathbb{Z}/2m\mathbb{Z} \oplus \mathbb{Z}/4n\mathbb{Z}: m=1,2\ \text{and}\ n=1,2\}.$
				\item $D= -19,-43,-67,-163,$ then  $\Phi_{K}(2,G)=  \{\mathbb{Z}/2\mathbb{Z} \oplus \mathbb{Z}/4n\mathbb{Z} : n=1,2,3\} \cup \{\mathbb{Z}/4\mathbb{Z} \oplus \mathbb{Z}/4\mathbb{Z}\}.$
			\end{enumerate}
			\item If $G \simeq \mathbb{Z}/2\mathbb{Z} \oplus \mathbb{Z}/2\mathbb{Z},$ then $\Phi_K(2,G) \subseteq \{\mathbb{Z}/2\mathbb{Z} \oplus \mathbb{Z}/{2n}\mathbb{Z} : n=1,2,3,4,5,6,8\} \cup \{\mathbb{Z}/4\mathbb{Z} \oplus \mathbb{Z}/4\mathbb{Z}\}.$
	\end{enumerate}	
\end{thm}

\subsection*{Outline of the paper} In section 2, we present a number of known results that narrow our focus. For example, an elliptic curve with full $2$-torsion over a number field $K$ admits a Weierstrass model of the form $E(\alpha,\beta): y^2 = x(x + \alpha)(x + \beta)$ where $\alpha,\beta \in \mathcal{O}_K.$ The existence of a point of order $3$, $4$ or $8$ is equivalent to finding solutions to certain equations involving $\alpha$ and $\beta$ (see Theorem \ref{newman}).

In section 3, we provide Theorem \ref{isogeny} which will be useful for section 4. Specifically, for a given $N,$ we classify certain $K$-rational cyclic $N$-isogenies for any $K$ in $\mathcal{S}$. In general,  the set of $K$-rational points of the modular curve $X_0(N)$ parametrizes the isomorphism classes $[(E,C)]$ of pairs $(E,C)$ where $E$ is an elliptic curve defined over $K$ and $C \subseteq E(\overline{K})$ is a Gal$(\overline{K}/K)$-invariant cyclic subgroup of order $N.$ If $X_0(N)$ has genus at least 2, then there are only finitely many $K$-rational points by a renowned result of Faltings. The $N$-values of interest are $40$ and $48.$ In both cases, $X_0(N)$ is a hyperelliptic curve with genus 3. The main trick for studying quadratic points on these curves is to instead study the rational points on its symmetric square $X_0(N)^{(2)}.$ This boils down to finding the rational points of its Jacobian $J_0(N).$ In light of the fact that $J_0(N)(\mathbb{Q})$ has rank 0 for the aforementioned values of $N,$ Bruin and Najman \cite{BN15} classified all exceptional quadratic points of $X_0(N).$ We complement their work and show that $X_0(N)(K)$ has no non-exceptional points for fields $K$ from $\mathcal{S}.$

In section 4, we prove the key ingredient Theorem \ref{mainthm1} which is the description of $T_K(G)$: the set of all possible torsion subgroups (up to isomorphism) for $E^d(K)_{\text{tor}}$ as $d$ varies over $K^*/(K^*)^2$ where $E/K$ is any elliptic curve with $E(K)_{\text{tor}} \simeq G.$ We study the torsion structures on quadratic twists because if $L=K(\sqrt{d})$ and $E/K$ is an elliptic curve then $E^d(K)_{\text{tor}}$ sets a bound for $E(L)_{\text{tor}}$ (see Proposition \ref{mainprop1}).

In section 5, we prove the main result Theorem \ref{mainthm2} of this paper. With the description of $T_K(G)$ at hand,  we can have a finite list of possible torsion groups for $E(L)_{\text{tor}}.$ Moreover, the growth over $L$ occurs if and only if there exists a point $P \in E(K)_{\text{tor}}$ with $P \in 2E(L)_{\text{tor}} \backslash\ 2E(K)_{\text{tor}}.$ The criterion for $2$-divisibility is given in Lemma \ref{lem1}. Applying this criterion yields systems of Diophantine equations. The difficulty arises with torsion points of high order such as $10$ or $12$. In order to work around this issue, we shift our perspective and benefit from using the action of Gal$(L/K)$ on $E(L)_{\text{tor}}.$ 

\subsection*{Notation} Throughout, we will use the following notations:
\begin{align*}
	\mathcal{S}=&\{\mathbb{Q}(\sqrt{D}) : D= -2,-7,-11,-19,-43,-67,-163\} \\
	\mathcal{T} =\{\mathbb{Q}&(\sqrt{D}) : D=-1,-2,-3,-7,-11,-19,-43,-67,-163\}
\end{align*} 
where $\mathcal{T}$ is the set of all imaginary quadratic fields with class number 1. Unless stated otherwise, $K \in \mathcal{S}$ is fixed and $L/K$ stands for any quadratic extension.  We denote the kernel of the multiplication by $n$ map on an elliptic curve $E$ by $E[n].$ Given an elliptic curve $E/K$ with $E(K)[2] \simeq \mathbb{Z}/2\mathbb{Z} \oplus \mathbb{Z}/2\mathbb{Z},$ we work with the Weierstrass model $E(\alpha,\beta): y^2=x(x+\alpha)(x+\beta)$ with $ \alpha,\beta \in \mathcal{O}_K$ for $E$ and use the database \cite{lmfdb} for rational elliptic curves. Some of the proofs in this paper rely on computations in Magma \cite{Magma} and Sage \cite{sage}. The code can be found $\href{https://github.com/irmak-balcik/NonCyclicTorsion}{\bf{here}.}$

\subsection*{Acknowledgement} The author is indebted to her advisor Sheldon Kamienny for suggesting this problem and his kind support.  The author is grateful to Andreas Schweizer for communicating a proof of the case $\mathbb{Z}/2\mathbb{Z} \oplus \mathbb{Z}/{32}\mathbb{Z}$ and allowing us to include it in the paper. We heartily thank {\"O}zlem Ejder for helpful discussions and the referees for their valuable comments, which greatly improved the results and exposition of the paper.

\begin{table}[h!]
	\caption{Examples of Growth of $G \simeq \mathbb{Z}/2\mathbb{Z} \oplus \mathbb{Z}/8\mathbb{Z}$ over $K=\mathbb{Q}(\sqrt{-7})$ }\label{table1}
	\centering		
	\begin{tabular}{| p{5.7cm} | p{4cm} | p{3cm}  | p{3cm} | }
		\hline
		Model for $E(\alpha,\beta),$ \ $w=\sqrt{-7}$ &$E^{d}(K)_{\text{tor}}$ & $d \in K$  & $E(K(\sqrt{d}))_{\text{tor}}$    \\ 
		\hline
		$(729,2304)$  & $\mathbb{Z}/2\mathbb{Z} \oplus \mathbb{Z}/{4}\mathbb{Z}$ & $-1$ & $\mathbb{Z}/4\mathbb{Z} \oplus \mathbb{Z}/{8}\mathbb{Z}$ \ \ \   \\ 
		\hline
		$\big(\frac{1}{2}(93w + 449),24w - 248\big)$  & $\mathbb{Z}/2\mathbb{Z} \oplus \mathbb{Z}/{2}\mathbb{Z}$ & $-15$ 
		& $\mathbb{Z}/{2}\mathbb{Z} \oplus \mathbb{Z}/{16}\mathbb{Z}$ \ 
		\\
		\hline
	\end{tabular} 
\end{table}

\begin{table}[h!]
	\caption{Example of Growth of $G \simeq \mathbb{Z}/2\mathbb{Z} \oplus \mathbb{Z}/6\mathbb{Z}$ over $K=\mathbb{Q}(\sqrt{D})$ }\label{tab3}
	\centering		
	\begin{tabular}{ |p{1.68cm} | p{9.1cm} | p{2.5cm} | p{2.5cm} | }
		\hline
		$ w=\sqrt{D}$ & Model for $E(\alpha,\beta) $ & $E^{-3}(K)_{\text{tor}}$  & $E(K(\sqrt{-3}))_{\text{tor}}$  \\
		\hline
		$D=-19$ &	 $\alpha=\frac{1160294229092597760}{806954491}w-\frac{755235215206514688}{806954491}$ 
		
		$\beta= -215373816000w + 140186761425$
		& $ \mathbb{Z}/2\mathbb{Z} \oplus \mathbb{Z}/4\mathbb{Z}$      &$\mathbb{Z}/2\mathbb{Z} \oplus  \mathbb{Z}/{12}\mathbb{Z} $  \\
		\hline
		$D=-43$ & $\alpha = -745604858793266380996608000000$ 
		
		$\beta=58753522825309371976852674194001$  & $\mathbb{Z}/2\mathbb{Z} \oplus \mathbb{Z}/4\mathbb{Z}$      & $\mathbb{Z}/2\mathbb{Z} \oplus  \mathbb{Z}/{12}\mathbb{Z} $  \\
		\hline
		$D=-67$ & $\alpha=-6172456047563616239669327589015552\times10^6$ 
		
		$\beta=384655287276782132832531502066904806161$ & $\mathbb{Z}/2\mathbb{Z} \oplus \mathbb{Z}/4\mathbb{Z}$      & $\mathbb{Z}/2\mathbb{Z} \oplus  \mathbb{Z}/{12}\mathbb{Z} $\\
		\hline
	\end{tabular} 
\end{table}

\begin{table}[h!]
	\caption{Example of Growth $G \simeq \mathbb{Z}/2\mathbb{Z} \oplus \mathbb{Z}/6\mathbb{Z} \ \text{over\ any}\ K \in  \mathcal{S}$}\label{tab2}
	\centering		
	\begin{tabular}{| p{5.7cm} | p{4cm} | p{3cm} | p{3cm}| }
		\hline
		Model for $E(\alpha,\beta)$  & $E^{d}(K)_{\text{tor}}$ & $d \in K$  & $E(K(\sqrt{d}))_{\text{tor}}$    \\ 
		\hline 
		$(64,189)$  & $\mathbb{Z}/2\mathbb{Z} \oplus \mathbb{Z}/{2}\mathbb{Z}$ & $21$ &  $\mathbb{Z}/2\mathbb{Z} \oplus \mathbb{Z}/{12}\mathbb{Z}$ \\
		\hline 
	\end{tabular} 
\end{table}

\begin{table}[h!]
	\caption{Examples of Growth of $G \simeq \mathbb{Z}/2\mathbb{Z} \oplus \mathbb{Z}/6\mathbb{Z}$ over $K=\mathbb{Q}(\sqrt{D})$ }\label{tab4}
	\centering		
	\begin{tabular}{ | p{1.68cm} | p{9.1cm} | p{2.5cm} |  p{2.5cm} | }
		\hline
		$ w=\sqrt{D}$ & Model for $E(\alpha,\beta)$  & $E^{-3}(K)_{\text{tor}}$  & $E(K(\sqrt{-3}))_{\text{tor}}$   \\ 
		\hline
		$D=-7$  &  $\big(\frac{1}{2}(21w - 39),\frac{1}{2}(-21w - 39)\big)$ &$\mathbb{Z}/2\mathbb{Z} \oplus \mathbb{Z}/6\mathbb{Z}$  & $\mathbb{Z}/6\mathbb{Z} \oplus \mathbb{Z}/{6}\mathbb{Z}$ \\ 
		\hline
		$D=-11$ & $(-19514w+101438,-6912w+214272)$
		& $\mathbb{Z}/2\mathbb{Z} \oplus \mathbb{Z}/6\mathbb{Z}$   & $\mathbb{Z}/6\mathbb{Z} \oplus \mathbb{Z}/6\mathbb{Z}$             \\
		\hline
		$D=-19$ & $\alpha=10993263062353152w-37976902494666157$ \
		
		$\beta  = -10437125916000000w+40849621725921875$ & $\mathbb{Z}/2\mathbb{Z} \oplus \mathbb{Z}/6\mathbb{Z}$ 
		&  $\mathbb{Z}/6\mathbb{Z} \oplus  \mathbb{Z}/6\mathbb{Z} $      \\
		\hline
	\end{tabular} 
\end{table}

\begin{table}[h!]
	\caption{Example of Growth $G \simeq \mathbb{Z}/2\mathbb{Z} \oplus \mathbb{Z}/4\mathbb{Z} \ \text{over\ any}\ K \in  \mathcal{S}$}\label{tab2x4-2x8}
	\centering		
	\begin{tabular}{| p{5.7cm} | p{4cm} | p{3cm} | p{3cm}| }
		\hline
		Model for $E(\alpha,\beta)$  & $E^{d}(K)_{\text{tor}}$ & $d \in K$  & $E(K(\sqrt{d}))_{\text{tor}}$    \\ 
		\hline 
		$(1,16)$  & $\mathbb{Z}/2\mathbb{Z} \oplus \mathbb{Z}/{2}\mathbb{Z}$ & $5$ &  $\mathbb{Z}/2\mathbb{Z} \oplus \mathbb{Z}/{8}\mathbb{Z}$ \\
		\hline 
	\end{tabular} 
\end{table}

\begin{table}[h!]
	\caption{Examples of Growth $G \simeq \mathbb{Z}/2\mathbb{Z} \oplus \mathbb{Z}/4\mathbb{Z} \ \text{over\ any}\ K = \mathbb{Q}(\sqrt{D}) \in \mathcal{S} $}\label{tab5}
	\centering		
	\begin{tabular}{ |p{5.7cm} | p{4cm} | p{3cm} | p{3cm} | }
		\hline
		$w=\sqrt{D}$ & Model for $E(\alpha,\beta)$  & $E^{-1}(K)_{\text{tor}}$  & $E(K(\sqrt{-1}))_{\text{tor}}$ \\
		\hline
		$D=-2$ & $(1, -8) $ &  $\mathbb{Z}/2\mathbb{Z} \oplus \mathbb{Z}/{4}\mathbb{Z}$ & $\mathbb{Z}/4\mathbb{Z} \oplus \mathbb{Z}/{4}\mathbb{Z}$ \\
		\hline
		$D=-7$& $(7569,7744)$  & $\mathbb{Z}/2\mathbb{Z} \oplus \mathbb{Z}/{4}\mathbb{Z}$ & $\mathbb{Z}/4\mathbb{Z} \oplus \mathbb{Z}/{4}\mathbb{Z}$ \\
		\hline
		$D=-11,-19,-43,-67,-163$ & $(16,25)$ &$\mathbb{Z}/2\mathbb{Z} \oplus \mathbb{Z}/{4}\mathbb{Z}$ & $\mathbb{Z}/4\mathbb{Z} \oplus \mathbb{Z}/{4}\mathbb{Z}$ \\
		\hline 
	\end{tabular} 
\end{table}

\begin{table}[h!]
	\caption{Examples of Growth $G \simeq \mathbb{Z}/2\mathbb{Z} \oplus \mathbb{Z}/2\mathbb{Z} \ \text{over some } K \in \mathcal{S}$} \label{tab88}
	\centering		
	\begin{tabular}{|p{2cm}| p{4cm} | p{4cm} | p{2cm} | p{3cm} | }
		\hline
		$w=\sqrt{D}$ &Model for $E(\alpha,\beta)$  & $E^{d}(K)_{\text{tor}}$  & $d \in K$ & $E(K(\sqrt{d}))_{\text{tor}}$    \\ 
		\hline
		$D=-11$ & $(2,-2)$ & $\mathbb{Z}/2\mathbb{Z} \oplus \mathbb{Z}/{2}\mathbb{Z}$ & $2$ & $\mathbb{Z}/2\mathbb{Z} \oplus \mathbb{Z}/4\mathbb{Z}$ \\
		\hline
		$D=-2$&$(-1,-2)$  & $\mathbb{Z}/2\mathbb{Z} \oplus \mathbb{Z}/{2}\mathbb{Z}$ & $-1$ & $\mathbb{Z}/4\mathbb{Z} \oplus \mathbb{Z}/{4}\mathbb{Z}$ \\
		\hline
		$D=-2$ &$(400,405)$ &  $\mathbb{Z}/2\mathbb{Z} \oplus \mathbb{Z}/4\mathbb{Z}$ & $-5$ & $\mathbb{Z}/2\mathbb{Z} \oplus \mathbb{Z}/8\mathbb{Z}$ \\
		\hline
	\end{tabular} 
\end{table}

\section{Auxiliary Results}

Sarma and Saikia \cite{SS18} have determined the complete list of torsion groups arising as $E(K)_{\text{tor}}$ where $E$ is an elliptic curve defined over any $K$ in $\mathcal{S}.$ We combine their results into a standalone theorem as follows.

\begin{thm}[\cite{SS18}]\label{mainlist}
	Let $K\in \mathcal{S}$ be fixed and let $E/K$ be any elliptic curve. 
		\begin{enumerate}
			\item If $K=\mathbb{Q}(\sqrt{-2}),$ then $E(K)_{\text{tor}}$ is isomorphic to either one of the groups in \eqref{eq:mazurgroups}, $\mathbb{Z}/{11}\mathbb{Z}\ \text{or}\  \mathbb{Z}/2\mathbb{Z} \oplus \mathbb{Z}/{10}\mathbb{Z}.$
			\item If $K=\mathbb{Q}(\sqrt{-7}),$ then $E(K)_{\text{tor}}$ is isomorphic to either one of the groups in \eqref{eq:mazurgroups}, $\mathbb{Z}/{11}\mathbb{Z}, \mathbb{Z}/{14}\mathbb{Z} \ \text{or}\  \mathbb{Z}/{15}\mathbb{Z}.$
			\item If $K=\mathbb{Q}(\sqrt{-11}),$ then $E(K)_{\text{tor}}$ is isomorphic to either one of the groups in \eqref{eq:mazurgroups},
			$\mathbb{Z}/{14}\mathbb{Z},  \mathbb{Z}/{15}\mathbb{Z}\   \text{or}\   \mathbb{Z}/2\mathbb{Z} \oplus \mathbb{Z}/{10}\mathbb{Z}.$
			\item If $K=\mathbb{Q}(\sqrt{-19}),$ then $E(K)_{\text{tor}}$ is isomorphic to either one of the groups in \eqref{eq:mazurgroups},
			$\mathbb{Z}/{11}\mathbb{Z},  \mathbb{Z}/{2}\mathbb{Z} \oplus \mathbb{Z}/{10}\mathbb{Z} \ \text{or}\ \ \mathbb{Z}/2\mathbb{Z} \oplus \mathbb{Z}/{12}\mathbb{Z}.$
			\item If $K=\mathbb{Q}(\sqrt{-43}),$ then $E(K)_{\text{tor}}$ is isomorphic to either one of the groups in \eqref{eq:mazurgroups},
			$\mathbb{Z}/{11}\mathbb{Z}, \mathbb{Z}/{14}\mathbb{Z}, \mathbb{Z}/{15}\mathbb{Z}\  \text{or}\  \mathbb{Z}/2\mathbb{Z} \oplus \mathbb{Z}/{12}\mathbb{Z}.$
			\item If $K = \mathbb{Q}(\sqrt{-67})$ then $E(K)_{\text{tor}}$ is isomorphic to either one of the groups in \eqref{eq:mazurgroups},
			$\mathbb{Z}/{14}\mathbb{Z},  \mathbb{Z}/{15}\mathbb{Z}\  \text{or}\  \mathbb{Z}/2\mathbb{Z} \oplus \mathbb{Z}/{12}\mathbb{Z}.$
			\item If $K=\mathbb{Q}(\sqrt{-163}),$ then $E(K)_{\text{tor}}$ is isomorphic to either one of the groups in \eqref{eq:mazurgroups},
			$\mathbb{Z}/{14}\mathbb{Z},  \mathbb{Z}/{15}\mathbb{Z} \ \text{or}\ \mathbb{Z}/2\mathbb{Z} \oplus \mathbb{Z}/{12}\mathbb{Z}.$
		\end{enumerate} 
\end{thm}

The next result is the generalization of \cite[Proposition 1]{Kwo97}, in which $K=\mathbb{Q},$ to number fields. This allows us to bound the growth of $E(K)_{\text{tor}}$ upon quadratic base change. Note that $P-\sigma(P)=(x,y)$ in $E^d(K)$ is understood as $(x,y/\sqrt{d})$ in the following statement.

\begin{prop}\label{mainprop1}(\cite{Kwo97})
	Let K be a number field, $L=K(\sqrt{d})$ a quadratic extension of K and let $\sigma$ denote the generator of $Gal(L/K)$. Given any elliptic curve $E/K$, there exists a homomorphism $h$ defined by
	\begin{align*}
	 E(L)_{\text{tor}} &\xrightarrow{h} E^d(K)_{\text{tor}} \\
	 P &\mapsto P- \sigma(P) 
	\end{align*}
	with ker$(h)=E(K)_{\text{tor}}$ and it induces an injection
	$ E(L)_{\text{tor}} / E(K)_{\text{tor}} \hookrightarrow E^{d}(K)_{\text{tor}}.$
\end{prop}

It is already known that the odd-torsion part of $E(L)_{\text{tor}}$ is well understood by two torsion structures that can occur over $K.$

\begin{lem}(\cite{GJT14},Cor 4)\label{lem2}
	If n is an odd positive integer we have
	$$ E(K(\sqrt{d}))[n] \simeq E(K)[n] \oplus E^d(K)[n]. $$
\end{lem}

Next, we recall the 2-divisibility condition which plays a key role in studying torsion growth.

\begin{lem}(\cite{Kna92})\label{lem1}
	Let E be an elliptic curve over a field k with $char(k)\neq 2,3.$ Suppose E is given by 
	$$ y^2=(x-\alpha)(x-\beta)(x-\gamma) $$
	with $\alpha, \beta, \gamma \in k.$ For $P=(x,y)$ in $E(k)$, there exists a k-rational point $Q$ on E such that 2$Q=P$ if and only if $x-\alpha, x-\beta$, and $x-\gamma$ are all squares in k.  \par 
\end{lem}

The following theorem which is due Newman, is a variant of the result of Ono \cite{Ono96}.

\begin{thm}(\cite{New17}) \label{newman}
	Let $K$ be a number field and $E/K$ be an elliptic curve with full 2-torsion. Then $E$ has a model of the form $y^2=x(x+\alpha)(x+\beta)$ where $\alpha,\beta \in \mathcal{O}_K.$
	
	\begin{enumerate}
		\item $E(K)$ has a point of order 4 if and only if $\alpha,\beta$ are both squares, or $-\alpha, \beta-\alpha$ are both squares, or $-\beta,\alpha-\beta$ are both squares in $\mathcal{O}_K.$
		\item $E(K)$ has a point of order 8 if and only if there exists a $c \in \mathcal{O}_K$, $c\neq 0$ and a Pythagorean triple $(u,v,w)$ with $u^2+v^2=w^2$ such that 
		$$ \alpha = c^2u^4, \beta =c^2v^4 $$
		or we can replace $\alpha,\beta$ by $-\alpha,\beta-\alpha$ or $-\beta,\alpha-\beta$ as in the first case.
		\item $E(K)$ contains a point of order 3 if and only if there exists $a,b,c \in \mathcal{O}_K$ with $c\neq 0$ and $\frac{a}{b} \notin \{-2,-1,-\frac{1}{2},0,1\}$ such that $\alpha=a^3(a+2b)c^2$ and $\beta=b^3(b+2a)c^2.$
	\end{enumerate}
\end{thm}

\section{Certain Isogenies}
Given an elliptic curve $E$ over a number field $K$, we call a subgroup $C$ of $E(\overline{K})$ of order $N$ a  $\mathbf{K}$-$\mathbf{rational}$ $\mathbf{N}$-$\mathbf{isogeny}$ 
if there exists an elliptic curve $E'/K$ and an isogeny $\phi : E \rightarrow E'$ over $K$ such that $C=ker(\phi).$ Equivalently, an order $N$-subgroup $C$ of $E(\overline{K})$ is a $K$-rational $N$-isogeny if it is invariant under the action of Gal$(\overline{K}/K).$ Note that $C$ might be $K$-rational even though it contains points that are not.  In what follows, we collect a number of classical results on $K$-rational isogenies which will be often used.
\begin{prop}\label{mn}
    Let $K$ be a number field and $E/K$ an elliptic curve with $|E(K)_{\text{tor}}|=m.$ If $H \subseteq E^d(K)$ is a subgroup of odd order $n$ for $d\in K$ a non-square,  then $E$ has a $K$-rational $mn$-isogeny pointwise defined over $K(\sqrt{d}).$ 
\end{prop}
\begin{proof}
   With respect to a short Weierstrass model fixed for $E,$ there exists an isomorphism $\tau:E^d \rightarrow E$ defined by $\tau((x,y))=(x,y\sqrt{d}).$ Since $H$ is a subgroup pointwise defined over $K,$ the image $\tau(H)$ of $H$ forms a Gal$(\overline{K}/K)$-invariant subgroup of $E(K(\sqrt{d}))$ of odd order $n.$ Moreover, $\tau(H) \cap E(K)_{\text{tor}}$ is trivial since $\tau(H)$ has no order $2$ points. Hence, $C:=\tau(H) \oplus E(K)_{\text{tor}}$ is a Gal$(\overline{K}/K)$-invariant subgroup of $E(K(\sqrt{d})).$ Therefore, $C $ is a $K$-rational $mn$-isogeny of $E$ pointwise defined over $E(K(\sqrt{d})).$
\end{proof}

\begin{thm}\label{isogeny}
	Let $K$ be fixed in $\mathcal{S}$ and let $E/K$ be an elliptic curve. Then, $E(\overline{K})$ has no $K$-rational cyclic $N$-isogeny for $N=40, 48$. 
\end{thm}

\begin{proof}
The modular curves $X_0(40)$ and $X_0(48)$ are hyperelliptic curves of genus 3 and admit the following equations by \cite[Section 4.3]{Rov98}
\begin{align*}
y^2 &=f_{40}(x)= x^8 + 8x^6 - 2x^4 + 8x^2 +1  \\
y^2 &= f_{48}(x)= x^8 + 14x^4 + 1. 
\end{align*}

Given any $N \in \{40,48\},$ let $P$ be a non-cuspidal quadratic point of $X_0(N)$ such that $P= (x, \pm \sqrt{f_N(x)})$ with $x \in \mathbb{Q}.$ Since $f_{N}(x) >  0$ for all $x \in \mathbb{R},$ $P$ must be defined over a real quadratic field, from which it follows that there are no non-cuspidal points in $X_0(N)(K)$ with rational $x$-coordinate. On the other hand, all quadratic points of $X_0(N)$ with non-rational $x$-coordinate are completely classified in \cite[Table 11,15] {BN15} and none of these points are defined over a field $K$ from $\mathcal{S}.$ 
In conclusion, $X_0(N)(K) = X_0(N)(\mathbb{Q})$ which consists only of cusps, proving the statement.
\end{proof}

\section{Classification of Twists}
One benefit of Theorem \ref{newman} is that any elliptic curve $E/K$ with full 2-torsion has a model
$E(\alpha,\beta): y^2=x(x+\alpha)(x+\beta)$
where $\alpha,\beta \in \mathcal{O}_K$. Note that $E(\alpha,\beta)$ is isomorphic to $E(-\alpha,\beta-\alpha), E(-\beta,\alpha-\beta)$ via $ (x,y) \mapsto (x-\alpha,y), (x,y) \mapsto (x-\beta,y) $ respectively. We further assume that gcd$(\alpha,\beta) $ is square-free since $ E(\alpha,\beta) \simeq E(d^2 \alpha, d^2\beta)$ by replacing $y$ with $y/d^3$ and $x$ with $x/d^2.$ Let $E^d$ be a quadratic twist of $E$ such that $E^d: y^2=x(x+d\alpha)(x+d\beta)$ where $d \in K$ is a non-square. Note that $E^d(K)[2] \simeq E(K)[2]. $  This section is devoted to determining the list of all possible torsion groups that can occur as  $E^d(K)_{\text{tor}}$.

Before proceeding further, we need a technical lemma which is a variant of \cite[Proposition 14]{New17}. We prove without using Gr{\"o}bner basis so that the proof is reversible to recover examples of elliptic curves with certain torsion structure.

\begin{lem}\label{lem3}
	Given $K=\Q(\sqrt{D}) \in \mathcal{T}$ and a non-square $d \in K,$ let $f(D)$ count the number of tuples $(a,b,a_0,b_0,c_0)$ where $a,b,a_0,b_0,c_0 \in \mathcal{O}_K$ are non-zero, $\frac{a}{b}, \frac{a_0}{b_0} \notin \{-2,-1,-\frac{1}{2},0,1\}$ and the following system of equations holds:
	 \begin{align} \label{eq:sys10}
	\begin{cases} 
	da^3(a+2b)=a_0^3(a_0 + 2b_0)c_0^2 \\
	db^3(b+2a)=b_0^3(b_0 + 2a_0)c_0^2.  
	\end{cases} 
	\end{align}
If $D=-1,-2,-3,$ then $f(D)$ is zero.
\end{lem}

\begin{rem}
Note that $f(D)$ is either $0$ or $infinite.$ If there is one solution, then one can get infinitely many simply by rescaling. 
\end{rem}

\begin{proof}
	Suppose $f(D)$ is non-zero. Our goal is to show that $D\neq-1,-2,-3.$ Let the system of equations \eqref{eq:sys10} hold for some $a,b,a_0,b_0,c_0$ in $\mathcal{O}_K$ satisfying the hypothesis. Since $b,b_0,a+2b, a_0 + 2b_0 \neq 0,$ dividing the equations yields
	\begin{align*}
	\left(\frac{a}{b}\right)^3 \frac{a+2b}{b+2a} = \left(\frac{a_0}{b_0}\right)^3 \frac{a_0 + 2b_0}{b_0 + 2a_0} 
	\end{align*} 	
	Substituting $x:=\frac{a_0}{b_0}$ and $y:=\frac{a}{b}$ we get
	$ y^3\frac{y+2}{2y+1} = x^3 \frac{x+2}{2x+1} $ which is equivalent to 
	$$ x^4 + 2x^3+ 2yx^4 + 4yx^3 - y^4 -2y^3 -2xy^4-4xy^3 =0 $$
	The polynomial on the left hand side is factored into irreducible polynomials as follows:
	\begin{align}\label{product}
	(x-y)(2x^3y+2xy^3+x^3+y^3+5x^2y+5xy^2+2x^2+2y^2+2x^2y^2+2xy)
	\end{align}

	If $x=y$ then $a=kb$ and $a_0=kb_0$ for some $k \in K.$ Plugging this in the system \eqref{eq:sys10} implies $d$ is a square, a contradiction. It follows that there exists a $K$-rational point $(x,y$) with $x,y \notin \{-2,-1,-\frac{1}{2},0,1\}$ lying on the curve $C$ defined by $p_C(x,y)=0$  where $p_C(x,y)$ denotes the second factor of \eqref{product}. Write $\overline{C}$ for the projective closure of $C$. By Magma, $\overline{C}$ is birationally equivalent (over $\mathbb{Q}$) to the elliptic curve $E_C$ whose Weierstrass model is given by 
	$$ y^2+2xy+2y = x^3 - x^2 - 2x$$
	with the following map
	\begin{align*}
	\phi : \overline{C} &\longrightarrow E_C \\
	   [x,y,z] &\longmapsto [p_1^{\phi}, p_2^{\phi},p_3^{\phi} ] 
	\end{align*}
	where 
	\begin{align*}
	 p_1^{\phi}&=2x^2y^2 + 3x^2yz + 4xy^2z -y^3z + x^2z^2 + 6xyz^2 -y^2z^2 + 2xz^3 \\
	p_2^{\phi} &= 2x^2y^2 + 4xy^3 - x^2yz + 10xy^2z + 3y^3z - x^2z^2 + 7y^2z^2 - 2xz^3 + 2yz^3 \\
	 p_3^{\phi} &= y^4 + 3y^3z + 3y^2z^2 + yz^3.
	\end{align*}
	
	Since  $\overline{C}(K)$ maps to $E_C(K)$, the only potential points in $\overline{C}(K)$ belong to the union of $\phi^{-1}(E_C(K))$ and the set of non-regular points of $\phi.$ So $\overline{C}(K)$ and $E_C(K)$ differ by at most finitely many points.

 Suppose $p_3^{\phi}(x,y,z)=0.$ If $z=0$ then $y=0$ and if $z=1$ then $y=0$ or $-1.$ Thus a non-regular point can be in the form $[x,0,0],[x,0,1]$ or $[x,-1,1]$ for $x \in K.$ Setting $p_2^{\phi}(x,y,z) =0$ narrows down the set of potential non-regular points of $\phi$ to $\{[1,0,0],[0,0,1],[-2,0,1],[-1,-1,1]\}.$ But none of those coming from $C$ gives rise to a point $(x,y)$  satisfying that $x,y$ are non-zero with $x,y \notin \{-2,-1,-\frac{1}{2},0,1\}.$

 To determine $\phi^{-1}(E_C(K))$, we compute by Sage that $E_C(K)$ has rank 0 for $D=-1,-2,-3$ and rank 1 for $D=-7,-11,-19,-43,-67,-163$ as well as their generators (see Table \ref{tabletable}). Hence, $\overline{C}(K)$ has infinitely many points so does $C(K)$ for $D\neq -1,-2,-3.$ It remains to describe $\phi^{-1}(E_C(K))$ explicitly for $D=-1,-2,-3.$

 	\begin{table}[h!]
 		\caption{K-rational points on $E_C$ }
 		\label{tabletable}
 		\centering		
 		\begin{tabular}{| p{1.9cm} | p{2.3cm} | p{12.1cm} |}
 					\hline 
 					$K=\mathbb{Q}(w)$ & $E_C(K)$ &  Generators of $E_C(K)$ (in projective coordinates) \\ 
 					\hline
 					$\mathbb{Q}(\sqrt{-1})$  &  $\mathbb{Z}/6\mathbb{Z}$ & $[2,-6,1]$\\ 
 					\hline
 					$\mathbb{Q}(\sqrt{-2})$ &  $\mathbb{Z}/6\mathbb{Z}$ &  $[2,-6,1]$\\
 					\hline
 					$\mathbb{Q}(\sqrt{-3})$ & $\mathbb{Z}/6\mathbb{Z} \oplus \mathbb{Z}/2\mathbb{Z}$ & $[2,-6,1],[1/2(-w+1),1/2(w-3),1]$\\
 					\hline
 					$\mathbb{Q}(\sqrt{-7})$  & $\mathbb{Z}/6\mathbb{Z} \oplus \mathbb{Z}$ &  $[2,-6,1],[-2,-w+1,1]$\\
 					\hline
 					$\mathbb{Q}(\sqrt{-11})$ &  $\mathbb{Z}/6\mathbb{Z} \oplus \mathbb{Z}$    & $[2,-6,1],[1/25(w - 17), 1/125(-9w - 147),1]$ \\
 					\hline
 					$\mathbb{Q}(\sqrt{-19})$ & $\mathbb{Z}/6\mathbb{Z} \oplus \mathbb{Z}$   & $[2,-6,1],[-25/9,1/27(28w+48),1]$      \\ \hline
 					$\mathbb{Q}(\sqrt{-43})$ &               $\mathbb{Z}/6\mathbb{Z} \oplus \mathbb{Z}$ & $[2,-6,1],[-15289/13689, 1/1601613(153160w + 187200),1]$
 					\\
 					\hline
 					$\mathbb{Q}(\sqrt{-67})$ &  $\mathbb{Z}/6\mathbb{Z} \oplus \mathbb{Z} $   & $[2,-6,1],$
 					
 					$[-2398489/2350089,1/3602686437(110531740w +74197200),1]$ \\
 					\hline
 					$\mathbb{Q}(\sqrt{-163})$ &  $\mathbb{Z}/6\mathbb{Z} \oplus \mathbb{Z}$    & $[2,-6,1], [-4220790466489/4220078644089,1/8669235817215083187$
 					
 					$(15276006029948680w +1462284655339200),1]$ \\
 		 \hline
 			\end{tabular} 
 \end{table}
 
 In order to compute $\phi^{-1}(E_C(K))$, we use the birational inverse $\psi$ of $\phi$ where
	\begin{align*}
	\psi : E_C &\longrightarrow \overline{C} \\ 
  [x,y,z] &\longmapsto [p_1^{\psi}, p_2^{\psi},p_3^{\psi}]
   \end{align*}
   such that
	\begin{align*}
	 p_1^{\psi} &=x^4 + 2x^3y - 4xy^2z - 2y^3z - 3x^2z^2 - y^2z^2 - 2xz^3 + 2yz^3 \\
	p_2^{\psi} &= -x^4 - 8x^3z - 4x^2yz - 21x^2z^2 - 14xyz^2 - 3y^2z^2 - 22xz^3 -10yz^3 - 8z^4  \\
	p_3^{\psi} &= x^4 + 6x^3z + 2x^2yz + 9x^2z^2 - 2xyz^2 - 3y^2z^2 + 4xz^3 -4yz^3. 
	\end{align*}
	
A quick computation in Magma shows that if $D=-1,-2,-3$ then each point of $\phi^{-1}(E_C(K))$ either does not correspond to a point $(x,y)$ on $C$ or fails to satisfy $x,y \notin \{-2,-1,-\frac{1}{2},0,1\}$ for which one can see Table \ref{table7}. Therefore, $C(K)$ has no points $(x,y)$ satisfying the hypothesis for $D=-1,-2,-3,$ completing the proof.
	
	\begin{table}[h!]
		\caption{Determination of $\phi^{-1}(E_C(K))$} \label{table7}
		\centering		
		\begin{tabular}{| p{7.4cm} | p{9.8cm}| }
			\hline 
			$P \in E_C(K)$ & $\phi^{-1}(P)$  \\ 
			\hline
			$[0,1,0]$  &  $[1,-1,1]$ \\ 
			\hline
			
			$[-1,0,1]$ &  $[-1,1,1]$  \\
			\hline
			
			$[0,-2,1]$ & $[-2,0,1]$   \\
			\hline
			
			$[0,0,1]$  & $[0,1,0]$ \\
			\hline
			
			$[2,-6,1]$ & $ [1,0,0]$    \\
			\hline
			
			$[2,0,1]$ & $[0,-2,1]$      \\    
			\hline
			
			[$-\sqrt{-3}-1$,$\sqrt{-3}-3$,1] & $[\sqrt{-3}-1,2,0]$ 
			 \\
			\hline
			
			[$\sqrt{-3}-1$,$-\sqrt{-3}-3$,1] &  $[-\sqrt{-3}-1,2,0]$ 
			\\
			\hline
			
			$[-\sqrt{-3}-1,\sqrt{-3}+3,1]$  &  $[0,0,1] $   \\
			\hline
			
			[$\sqrt{-3}-1$,$-\sqrt{-3}+3$,1]  & $[0,0,1]$      \\
			\hline
			
			[$\frac{1}{2}(-\sqrt{-3}+1),\frac{1}{2}$($\sqrt{-3}-3),1]$ & $[-1,-1,1]$   \\
			\hline
			
			[$\frac{1}{2}(\sqrt{-3}+1),\frac{1}{2}(-\sqrt{-3}-3),1]$ & $[-1,-1,1]$       \\
			\hline
			
			[-2,$-\sqrt{-7}+1,1]$  & $[\frac{1}{8}(-3\sqrt{-7}-1),\frac{1}{4}(\sqrt{-7}+3),1]$   \\
			\hline
			
			$[\frac{1}{25}(\sqrt{-11}-17),\frac{1}{125}(-9\sqrt{-11}-147),1]$ & $[\frac{1}{24}(-\sqrt{-11}-43),\frac{1}{24}(19-\sqrt{-11}),1]$  \\
			\hline
			
			$[-\frac{25}{9}, \frac{1}{27}(28\sqrt{-19} + 48),1]$ &  $[\frac{1}{21457}(4536\sqrt{-19} + 8335), 
			 \frac{1}{12475}(-2688\sqrt{-19} + 4283),1]$     \\
			 
			\hline
			 $[-\frac{15289}{13689}, \frac{1}{1601613}(153160\sqrt{-43} + 187200),1] $       &$[\frac{1}{24095171196721}(1471818282480\sqrt{-43}-22077784435121),$
			         
			         $\frac{1}{8634098306107}(-57343104000\sqrt{-43} + 8625906306107), 1]$ \\		
			\hline
		\end{tabular} 
	\end{table}
\end{proof}	

\begin{rem}
	By using the additional data provided in Table \ref{table7}, we can work backwards in the proof of Lemma \ref{lem3} to produce examples of elliptic curves with certain torsion structure for $D=-7,-11,-19$. As the coefficients become overwhelmingly long for $D=-43,-67,-163$, an explicit example is not included in Table \ref{tab4}.
	\end{rem}

\begin{thm}\label{mainthm1}
	Given $K$ in $\mathcal{S},$ let $G \in \Phi_K(1)$ such that $\mathbb{Z}/2\mathbb{Z} \oplus \mathbb{Z}/2\mathbb{Z} \subseteq G.$
	
	\begin{enumerate}
		
		\item If $G \simeq \mathbb{Z}/2\mathbb{Z} \oplus \mathbb{Z}/{12}\mathbb{Z},$  then $T_K(G)=\{\mathbb{Z}/2\mathbb{Z} \oplus \mathbb{Z}/{2}\mathbb{Z}\}.$

		\item If $G \simeq \mathbb{Z}/2\mathbb{Z} \oplus \mathbb{Z}/{10}\mathbb{Z},$  
		then $T_K(G)=\{\mathbb{Z}/2\mathbb{Z} \oplus \mathbb{Z}/2\mathbb{Z}\}.$ 
		
		\item If $G \simeq \mathbb{Z}/2\mathbb{Z} \oplus \mathbb{Z}/8\mathbb{Z}$ when \begin{enumerate}[i.]
			\item $D \neq -7,$ then $T_K(G)=\{\mathbb{Z}/2\mathbb{Z} \oplus \mathbb{Z}/2\mathbb{Z}\}.$ 
			\item $D=-7,$ then $T_K(G)=\{\mathbb{Z}/2\mathbb{Z} \oplus \mathbb{Z}/{2}\mathbb{Z}, \mathbb{Z}/2\mathbb{Z} \oplus \mathbb{Z}/4\mathbb{Z} \}.$
		\end{enumerate}	
		\item	If $G \simeq \mathbb{Z}/2\mathbb{Z} \oplus \mathbb{Z}/6\mathbb{Z}$ when
		
		\begin{enumerate}[i.]
			\item  $D=-2,$ then $T_K(G)=\{\mathbb{Z}/2\mathbb{Z} \oplus \mathbb{Z}/2\mathbb{Z}\}.$
			\item $D=-7,-11,$ then
			$T_K(G)=\{\mathbb{Z}/2\mathbb{Z} \oplus \mathbb{Z}/{2}\mathbb{Z}, \mathbb{Z}/2\mathbb{Z} \oplus \mathbb{Z}/6\mathbb{Z}\}.$
			\item $D=-19,-43,-67,-163,$ then $T_K(G)=\{\mathbb{Z}/2\mathbb{Z} \oplus \mathbb{Z}/{2n}\mathbb{Z} : n=1,2,3 \}.$
			\end{enumerate}
		
		\item If $G \simeq \mathbb{Z}/2\mathbb{Z} \oplus \mathbb{Z}/4\mathbb{Z}$ when 
		\begin{enumerate}[i.]
			\item	$D=-2,-11,$  then $T_K(G)=\{\mathbb{Z}/2\mathbb{Z} \oplus \mathbb{Z}/2\mathbb{Z},\mathbb{Z}/2\mathbb{Z} \oplus \mathbb{Z}/4\mathbb{Z}\}.$ 		
				\item $D=-7,$ then $T_K(G)=\{\mathbb{Z}/2\mathbb{Z} \oplus \mathbb{Z}/2\mathbb{Z}, \mathbb{Z}/2\mathbb{Z} \oplus \mathbb{Z}/4\mathbb{Z},\mathbb{Z}/2\mathbb{Z} \oplus \mathbb{Z}/8\mathbb{Z}\}.$ 
				\item $D=-19,-43,-67,-163,$ then $T_K(G)=\{\mathbb{Z}/2\mathbb{Z} \oplus \mathbb{Z}/{2n}\mathbb{Z}: n=1,2,3\}.$
			\end{enumerate}
		
		\item If $G \simeq \mathbb{Z}/2\mathbb{Z} \oplus \mathbb{Z}/2\mathbb{Z}$ then $T_K(G)\subseteq \{\mathbb{Z}/2\mathbb{Z} \oplus \mathbb{Z}/{2n}\mathbb{Z} : n=1,2,3,4,5,6\}.$  
	\end{enumerate}
\end{thm}

\begin{proof} 
Let $E$ denote an arbitrary elliptic curve over $K$ with $E(K)_{\text{tor}} \simeq G.$ Then $E$ has a model of the form  $y^2=x(x+\alpha)(x+\beta)$ for some $\alpha,\beta \in \mathcal{O}_K.$ Let $E^d$ denote the quadratic twist of $E$ by $d\in K$ a non-square. 
We already know from Theorem \ref{mainlist} that $T_K(G)$ must be a subset of $\{\mathbb{Z}/2\mathbb{Z} \oplus \mathbb{Z}/{2n}\mathbb{Z} \ : \ 1 \leq n \leq 6\}.$ 
\vspace{.2cm}

$\textit{(1)}$ Suppose $E(K)_{\text{tor}} \simeq \mathbb{Z}/2\mathbb{Z} \oplus \mathbb{Z}/{12}\mathbb{Z}.$ Note that $\mathbb{Z}/3\mathbb{Z} \nsubseteq E^d(K)_{\text{tor}}$ as $\mathbb{Z}/3\mathbb{Z} \oplus \mathbb{Z}/12\mathbb{Z}$ cannot be realized over quartic fields by \cite[Theorem 8]{BN16}.  Also, $\mathbb{Z}/5\mathbb{Z} \not\subseteq E^d(K)_{\text{tor}},$ otherwise together with a $K$-rational $12$-torsion point, $E$ would have a $K$-rational cyclic 60-isogeny pointwise defined over $K(\sqrt{d})$ from Proposition \ref{mn}.  Equivalently, the modular curve $X_0(60)$ has a non-cuspidal $K$-rational point, which is not possible by \cite[9.1]{BrN}. Moreover, $\mathbb{Z}/{8}\mathbb{Z} \not\subseteq E^d(K)_{\text{tor}}.$ Otherwise, replacing $E$ by $E^d$ in Proposition \ref{mn} if necessary, we may assume that $E$ has two independent $K$-rational cyclic isogenies of order $2$ and $24$. By \cite[Lemma 7]{Naj16}, $E$ is $K$-isogenous to an elliptic curve $E'/K$ with a cyclic $48$-isogeny, contradicting Theorem \ref{isogeny}.

It remains to show that $E^d(K)$ has no point of order $4$. By way of contradiction, let $E^d(K)\simeq \mathbb{Z}/2\mathbb{Z} \oplus \mathbb{Z}/4\mathbb{Z}.$
There exists a short exact sequence 
\begin{align}\label{exact1}
	0 \rightarrow \text{ker}(\psi)\xrightarrow{i} E(K(\sqrt{d})) \xrightarrow{\psi} E(K) \times E^d(K) \xrightarrow{\pi} \text{coker}(\psi) \rightarrow 0 
\end{align}
where $\psi$ maps $R = (x,y)$ to  
$$\psi(R) = (R+\sigma(R), \phi(R-\sigma(R),1/\sqrt{d}))\ \text{with}\ \phi(R,a)=(x,ay).$$ 

Note that ker$(\psi)=\mathbb{Z}/2\mathbb{Z} \oplus \mathbb{Z}/2\mathbb{Z}$ since Gal$(K(\sqrt{d})/K)$ fixes a basis of $E[2]$. Restricting $\psi$ to the torsion part and applying \cite[Theorem 8]{BN16} imply that $E(K(\sqrt{d}))_{\text{tor}} \simeq \mathbb{Z}/2\mathbb{Z} \oplus \mathbb{Z}/12\mathbb{Z}$ or $\mathbb{Z}/2\mathbb{Z} \oplus \mathbb{Z}/24\mathbb{Z}.$ In both cases, coker$(\psi)$ must contain a point of order $4$ which gives a contradiction due to \cite[Theorem 3]{GJT14}.
\vspace{.2cm}

$\textit{(2)}$ Suppose $E(K)_{\text{tor}}\simeq \mathbb{Z}/2\mathbb{Z} \oplus \mathbb{Z}/{10}\mathbb{Z}.$ Since  $T_K(G) \subseteq \{\mathbb{Z}/2\mathbb{Z} \oplus \mathbb{Z}/{2n}\mathbb{Z} \ : \ 1 \leq n \leq 6\}$, we will eliminate the non-trivial cases for $n.$

\vspace{.1cm}

	$n \in \{3,6\}$:
	In both cases, $E^d(K)_{\text{tor}}$ has a point of order $3$. By Proposition \ref{mn}, $E(K(\sqrt{d}))$ has two independent Gal$(\overline{K}/K)$-invariant cyclic subgroups of order $2$ and $30.$ It follows from \cite[Lemma 7]{Naj16} that $E$ is $K$-isogenous to an elliptic curve with a cyclic $60$-isogeny. Equivalently, the modular curve $X_0(60)$ has a non-cuspidal $K$-rational point which is not possible by \cite[9.1]{BrN}.	
\vspace{.1cm}

	$n \in\{2,4\}$:
	In both cases, $E^d(K)_{\text{tor}}$ has a point of order $4$. By Proposition \ref{mn}, $E(K(\sqrt{d}))$ has two independent Gal$(\overline{K}/K)$-invariant cyclic subgroups of order $2$ and $20.$ By \cite[Lemma 7]{Naj16} $E$ is isogenous over $K$ to an elliptic curve $E'/K$ with a cyclic $40$-isogeny which in turn implies that $X_0(40)(K)$ has a non-cuspidal point, a contradiction by Theorem \ref{isogeny}.
\vspace{.1cm}
	
	$n = 5$:
	In this case, $E^d(K)_{\text{tor}}$ has a point of order $5$. By Lemma \ref{lem2}, $E[5]$ is contained in $E(K(\sqrt{d}))$  which implies by the Weil pairing that $K(\sqrt{d})=\mathbb{Q}(\zeta_{5})$ where $\zeta_5$ is a primitive $5$th root of unity. Then  Gal$(\mathbb{Q}(\zeta_{5})/\mathbb{Q})$ is isomorphic to $\mathbb{Z}/{4}\mathbb{Z}$ and so every intermediate field of $\mathbb{Q}(\zeta_{5})$ is totally real, a contradiction. 
\vspace{.2cm}

$\textit{(3)}$	Suppose $E(K)_{\text{tor}} \simeq \mathbb{Z}/2\mathbb{Z} \oplus \mathbb{Z}/8\mathbb{Z}$. Without loss of generality, we assume that there exists a triple $(x,y,z)\in \mathcal{O}_K$ such that $\alpha = x^4,$ $\beta=y^4$ and $x^2+y^2=z^2.$ If $E^d(K)$ has a point of order $8,$ then by the exact sequence \eqref{exact1},  $E(K(\sqrt{d}))_{\text{tor}}/\text{ker}(\psi)$ injects into $E(K)_{\text{tor}} \times E^d(K)_{\text{tor}}$ where ker$(\psi)= \mathbb{Z}/2\mathbb{Z} \oplus \mathbb{Z}/2\mathbb{Z}.$ Using \cite[Theorem 8]{BN16} reduces the possibility for $E(K(\sqrt{d}))_{\text{tor}}$ to $\mathbb{Z}/4\mathbb{Z} \oplus \mathbb{Z}/8\mathbb{Z}$ and $\mathbb{Z}/2\mathbb{Z} \oplus \mathbb{Z}/16\mathbb{Z}.$ In each case, coker$(\psi)$ contains at least one point of order $4$, which yields a contradiction due to \cite[Theorem 3]{GJT14}.

If $E^d(K)$ has a point of order $5$, then together with a $K$-rational $8$-torsion point on $E$ and by Proposition \ref{mn}, there exists a $K$-rational cyclic $40$-isogeny of $E$ pointwise defined over $K(\sqrt{d}).$ But this cannot happen by Theorem \ref{isogeny}.

If $E^{d}(K)_{\text{tor}}$ has a point of order $4$, then  by Theorem \ref{newman} one of the following is satisfied:
\begin{enumerate}[(i)]

		\item $dx^4,$ $dy^4$ are squares, 
		\item $-dx^4,$ $dy^4-dx^4$ are squares,
		\item $-dy^4,$ $dx^4-dy^4$ are squares.
\end{enumerate}
In case (i), $d$ must be a square, a contradiction. In case (ii) or (iii), it follows that $d=-u^2$ for some $u \in \mathcal{O}_K.$ Since $dy^4-dx^4=\pm v^2$ for some $v \in \mathcal{O}_K,$ then dividing by $d$ yields
	$$ y^4-x^4=\mp (vu^{-1})^2$$
and multiplying by $-1$ if necessary, we get $x^4-y^4=w^2$ for some $w \in \mathcal{O}_K.$ Note that $x^4,y^4,x^4-y^4 \neq 0$ since $E$ is smooth and so $x,y,w \neq 0$. By \cite[Proposition 11]{New17}, there is no nontrivial solution to $x^4-y^4=w^2$ except $D=-7$ (see Table \ref{table1}).

	If $E^d(K)_{\text{tor}}$ has a point of order $3$, then by Proposition \ref{mn} the subgroups $\mathbb{Z}/2\mathbb{Z}$ and $\mathbb{Z}/{24}\mathbb{Z}$ of $E(K(\sqrt{d}))$ form two independent $K$-rational isogenies of $E.$ It follows from  \cite[Lemma 7]{Naj16} that $E$ is $K$-isogenous to an elliptic curve $E'/K$ with a cyclic $48$-isogeny, contradicting Theorem \ref{isogeny}. 

\vspace{.2cm}

$\textit{(4)}$	Suppose $E(K)_{\text{tor}} \simeq \mathbb{Z}/2\mathbb{Z} \oplus \mathbb{Z}/6\mathbb{Z}$. We assume that there exists $a, b, c \in \mathcal{O}_{K}$ with $c \neq 0$ and $\frac{a}{b} \not\in  \{-2,-1,-\frac{1}{2},0,1\}$ such that $\alpha = a^3(a + 2b)c^2$ and $\beta = b^3(b + 2a)c^2$. By choosing $E^{c^2}$ instead of $E$ if necessary, we further assume $\alpha=a^3(a + 2b)$ and $\beta=b^3(b + 2a).$  Note that $E^d(K)$ cannot have a point of order $8,10,$ and $12$ by the proof of the previous assertions.	It remains to argue whether $E^d(K)$ has a point of order $3$ and $4$. Suppose  $E^d(K)\simeq \mathbb{Z}/2\mathbb{Z} \oplus \mathbb{Z}/6\mathbb{Z}.$ Then we have $d=-3$ by the Weil pairing. On the other hand, there are $a_0,b_0,c_0 \in \mathcal{O}_{K}$ with $c_0 \neq 0$ and $\frac{a_0}{b_0} \notin \{-2,-1,-\frac{1}{2},0,1\}$ such that 	

$$ -3a^3(a+2b)= a_0^3(a_0 + 2b_0)c_0^2 \ \ \  -3b^3(b+2a)= b_0^3(b_0 +2a_0)c_0^2. $$

As $-3$ is a non-square in $K,$ it follows from Lemma \ref{lem3} that  $D\neq-2.$ One can see Table \ref{tab4} for examples.  If $ E^d(K)_{\text{tor}} \simeq  \mathbb{Z}/2\mathbb{Z} \oplus \mathbb{Z}/4\mathbb{Z},$ then by Theorem \ref{newman} one of the following holds:
\vspace{.1cm}

		 (i) $d\alpha, d\beta$ are squares, 
		 (ii) $-d\alpha, d\beta-d\alpha$ are squares, or
		 (iii) $-d\beta,d\alpha-d\beta$ are squares. 

With no loss of generality, we assume (i) holds, i.e. $d\alpha = s^2,d\beta = t^2$ for some $s,t \in \mathcal{O}_K$. Note that $s,t \neq  0,s^2 \neq  t^2.$ We have
	$$ s^2 = d\alpha = da^3(a + 2b) \  \  \ t^2 = d\beta = db^3(b + 2a).
	$$
	Multiplying by $d$ we get
	$$ ds^2 = a^3(a + 2b)d^2 \ \ \ dt^2 = b^3(b + 2a)d^2.
	$$
	Taking the products and dividing both sides by $d^4a^2b^6$ along with substituting $z:=\frac{a}{b}$ yield
	$$
	\left(\frac{st}{dab^3}\right)^2 = z(z+2)(2z+1)
	$$
	Multiplying by $4$ gives
	$
	\left(\frac{2st}{dab^3}\right)^2 = 2z(2z+4)(2z+1).
	$ So that we obtain a $K$-rational point on the elliptic curve $E_0$ defined by 
	$$ 
	y^2 = x(x+4)(x+1)=x^3+5x^2+4x.
	$$
By Sage,
$E_0(K) \simeq \mathbb{Z}/2\mathbb{Z} \oplus \mathbb{Z}/4\mathbb{Z} =\{[0,1,0], [0,0,1], [-4,0,1], [-1,0,1], [-2,\pm 2,1],[2,\pm 6,1]\}$
when $D=-2,-7,-11.$ Then, $2z=-4,-2,-1,0$ or $2$ and so $z \in \{-2,-1,-\frac{1}{2},0,1\}$, a contradiction. On the other hand, $E_0(K)\simeq \mathbb{Z}/2\mathbb{Z} \oplus \mathbb{Z}/4\mathbb{Z} \oplus \mathbb{Z}$ when $D=-19,-43,-67,-163.$ For example, the extra points in $E_0(K)$ for $D=-19,-43,-67$ give rise to the torsion growth as shown in Table \ref{tab3}, establishing the assertion.
\vspace{.2cm}

$\textit{(5)}$	Suppose $E(K)_{\text{tor}} \simeq \mathbb{Z}/2\mathbb{Z} \oplus \mathbb{Z}/4\mathbb{Z}$. Replacing $E$ by $E^d$ in the proof of $\textit{(1)}$ above, we see that $E^{d}$ has no point of order $12.$ Similarly, the possibility of $E^d(K)_{\text{tor}}$ having a $10$-torsion point is dismissed by case $n=5$ in the proof of $\textit{(2)}.$

Notice that $E^d$ has no point of order $8$ except $D=-7$ by the proof of $\textit{(3)}.$ One can take $E$ as $E^{-1}(729,2304)$ in Table \ref{table1} for $D=-7.$ One can also see Table \ref{tab5} where $E^d(K)$ has a $4$-torsion point for any $K \in \mathcal{S}.$ By replacing $E$ by $E^d$ in the proof of $\textit{(4)}$ above, $E^d(K)$ can have a $3$-torsion point only when $D=-19,-43,-67,-163.$ For example, one can consider $E$ as $E(-\frac{\alpha}{3}, -\frac{\beta}{3})$ in Table \ref{tab3}. 
 \vspace{.2cm}
 
$\textit{(6)}$ It follows from \eqref{eq:kamiennygroups} and $E(K)[2] \simeq E^d(K)[2]$ for any $d\in K.$ 
	
\end{proof}

\section{Growth of Torsion}

Let $K$ be a number field and let $F$ denote the maximal elementary abelian $2$-extension of $K$ where
\begin{align} \label{ffff}
	F=K(\{\sqrt{d} : d\in \mathcal{O}_K\}). 
\end{align}

We need the following technical result prior to the proof of the main theorem (the result is in fact true for any number field $K$, after minor modifications of the proof).

\begin{lem}\label{lem7}
	Let $K$ be in $\mathcal{S}$  and $\alpha \in K.$ If $\sqrt{\alpha}$ (or $-\sqrt{\alpha})$ is a square in $F$ then $\alpha$ or $-\alpha$ is a square in $K.$ 
\end{lem}	

\begin{proof}
	Let $z \in F$ such that $\sqrt{\alpha}=z^2$ (or $\sqrt{\alpha}=-z^2).$ Then $z$ is a root of the polynomial $f(x)=x^4-\alpha$ defined over $K.$ Since $f$ has a root in $F$ and $F/K$ is Galois, $f(x)$ must split over $F.$ Let $K_f$ denote the splitting field of $f$ over $K.$ Notice that Gal$(K_f/K)$ is an elementary abelian $2$-group since Gal$(K_f/K)$ is a quotient of Gal$(F/K).$ 
	
	If $f(x)$ is not irreducible over $K$, then it is either a product of two quadratic polynomials or one quadratic with two linear polynomials. We will deal with both cases with the following argument. Let
	$$ f(x)=x^4-\alpha=(x^2+ax+b)(x^2+cx+d) $$
	for some $a,b,c,d \in K.$ Then $ac+b+d=0,$ $ad+bc=0$ and $a+c=0.$ Replacing $a$ by $-c$ gives $c(b-d)=0.$ So either $c=0$ or $b=d.$ If $c=0$ then $b=-d$ and so $bd=-\alpha$ which implies $\alpha=b^2$ in $K.$ If $b=d,$ then $\alpha=-bd,$ so $-\alpha=b^2$ in $K.$

	If $f(x)$ is irreducible over $K,$ then since an elementary abelian $2$-subgroup of $S_4$ is either of order $2$ or the Klein four-group $V_2,$ we have Gal$(K_f/K)\simeq V_2.$ Note that $i \not\in K.$ Then $f(x)$ is reducible over $K(i)$ since $K_f=K(z,i)$ and it has degree $4$ over $K.$ We know from the preceding paragraph that $\alpha$ or $-\alpha$ is a square in $K(i).$ If $\alpha=\pm d^2$ with $d \in K(i)$ then as $\alpha\in K,$ we have $d=bi$ for some $b\in K.$ Hence, $\alpha=\pm b^2$ as desired.     
\end{proof}

\subsection*{Proof of Theorem \ref{mainthm2}}

Given $K \in \mathcal{S}$ and $G \in \Phi_K(1)$ with $\mathbb{Z}/2\mathbb{Z} \oplus \mathbb{Z}/2\mathbb{Z} \subseteq G,$ we want to determine $\Phi_K(2,G).$ Let  $E/K$ be an arbitrary elliptic curve such that $E(K)_{\text{tor}}\simeq G$ and let $L$ be any quadratic extension of $K$ with \text{Gal}$(L/K)=\langle \sigma \rangle.$ 
\vspace{.2cm}

$\textit{(1)}$\  Suppose $E(K)_{\text{tor}} \simeq \mathbb{Z}/2\mathbb{Z} \oplus \mathbb{Z}/{12}\mathbb{Z}.$ It follows from Proposition \ref{mainprop1} and Theorem \ref{mainthm1} that  $E(L)_{\text{tor}}$ is isomorphic to one of the groups:
\begin{align*}
	\mathbb{Z}/2\mathbb{Z} \oplus \mathbb{Z}/12\mathbb{Z}, \ \mathbb{Z}/4\mathbb{Z} \oplus \mathbb{Z}/{12}\mathbb{Z},\ \mathbb{Z}/2\mathbb{Z} \oplus \mathbb{Z}/{24}\mathbb{Z}, \ \mathbb{Z}/4\mathbb{Z} \oplus \mathbb{Z}/{24}\mathbb{Z}.
\end{align*}
By \cite[Theorem 8]{BN16}, $\mathbb{Z}/4\mathbb{Z} \oplus \mathbb{Z}/{12}\mathbb{Z}$ cannot occur as a subgroup of elliptic curves over quartic number fields. So it remains to show that $E(L)_{\text{tor}}$ cannot be isomorphic to $\mathbb{Z}/2\mathbb{Z} \oplus \mathbb{Z}/24\mathbb{Z}.$

Suppose $E(L)_{\text{tor}} \simeq \mathbb{Z}/2\mathbb{Z} \oplus \mathbb{Z}/24\mathbb{Z}.$ Let $Q$ be a point of order 2 and $P$ a point of order 24 such that $E(L)_{\text{tor}} = \langle Q, P \rangle.$  Then every point $aP+bQ$ where $a$ is odd has order divisible by 8, and hence lies outside $E(K)_{\text{tor}}.$ So the 24 points $aP+bQ$ with even $a$ must all be in $E(K)_{\text{tor}}$ and hence fixed by $\sigma.$ In particular, this holds for $2P.$ Since $\sigma(P)=aP+bQ$ with $a \in (\mathbb{Z}/24\mathbb{Z})^{\times}$ we get 
$$ 
2P=\sigma(2P)=2aP+2bQ 
$$
from which it follows that $a=1$ or $13.$ If $b=0,$ then $E$ has a $K$-rational cyclic $24$-isogeny generated by $P.$ Applying \cite[Lemma 7]{Naj16} to the independent isogenies $\langle Q \rangle$ and $\langle P \rangle$ of $E$, we have that $E$ is $K$-isogenous to an elliptic curve $E'$ with a cyclic $48$-isogeny.
But this gives a contradiction to Theorem \ref{isogeny}. 
If $b=1,$ then $\sigma(P)=P+Q$ or $13P+Q.$ We consider the quotient  curve $\overline{E}:=E/ \langle Q \rangle,$ resp. $E/\langle 12P+Q \rangle$ with the isogeny $\phi: E \rightarrow \overline{E}.$ Writing $\overline{P}$ for the image of $P$ under $\phi$, we have $\sigma(\overline{P})=\overline{P}.$ So $\overline{E}$ has a $K$-rational $24$-torsion point $\overline{P},$ a contradiction by Theorem \ref{mainlist}.
\vspace{.2cm}

$\textit{(2)}$\	Suppose $E(K)_{\text{tor}} \simeq \mathbb{Z}/2\mathbb{Z} \oplus \mathbb{Z}/{10}\mathbb{Z}.$ By Proposition \ref{mainprop1} and Theorem \ref{mainthm1}, the only non-trivial groups that can occur as $E(L)_{\text{tor}}$ are 
	$ \mathbb{Z}/2\mathbb{Z} \oplus \mathbb{Z}/{20}\mathbb{Z}$ and $ \mathbb{Z}/4\mathbb{Z} \oplus \mathbb{Z}/{20}\mathbb{Z}.$ We will follow a similar argument as in the previous proof to exclude each group.

Suppose $E(L)_{\text{tor}} \simeq \mathbb{Z}/2\mathbb{Z} \oplus \mathbb{Z}/20\mathbb{Z}.$ Let $Q$ be a point of order 2 and $P$ a point of order 20 such that $E(L)_{\text{tor}} = \langle Q, P \rangle.$  Then every point $aP+bQ$ where $a$ is odd has order divisible by 4, and hence lies outside $E(K)_{\text{tor}}.$ So the 20 points $aP+bQ$ with even $a$ must all be in $E(K)_{\text{tor}}$ and hence fixed by $\sigma.$ In particular, this holds for $2P.$ Since $\sigma(P)=aP+bQ$ with $a \in (\mathbb{Z}/20\mathbb{Z})^{\times}$ we get 
$$ 
2P=\sigma(2P)=2aP+2bQ 
$$
from which it follows that $a=1$ or $11.$ If $b=0,$ then $E$ has a $K$-rational cyclic $20$-isogeny generated by $P.$ Together with the independent $K$-rational $2$-torsion point $Q,$ we can apply \cite[Lemma 7]{Naj16} from which it follows that $E$ is isogenous to an elliptic curve $E'/K$ with a cyclic $40$-isogeny. But this gives a contradiction to Theorem \ref{isogeny}. 
If $b=1,$ then $\sigma(P)=P+Q$ or $11P+Q.$ We consider the elliptic curve $\overline{E}:=E/ \langle Q \rangle$ resp. $E/\langle 10P+Q \rangle$ with the isogeny $\phi: E \rightarrow \overline{E}.$ Writing $\overline{P}$ for the image of $P$ under $\phi,$ we have $\sigma(\overline{P})=\overline{P}.$ So $\overline{E}$ has a $K$-rational $20$-torsion point $\overline{P}$, a contradiction by Theorem \ref{mainlist}.

With a similar argument, we will eliminate the possibility where $E(L)_{\text{tor}} \simeq \mathbb{Z}/4\mathbb{Z} \oplus \mathbb{Z}/20\mathbb{Z}.$ Suppose $E(L)_{\text{tor}} =\langle Q',P \rangle$ where $Q'$ is a point of order $4$ and $P$ is a point of order $20.$ Then every point $aP+bQ'$ with odd $a$ or odd $b,$ has order divisible by 4 and hence lies outside $E(K)_{\text{tor}}.$ So the 20 points $aP+bQ'$ where $a,b$ are both even, must all be in $E(K)_{\text{tor}}.$ In particular, $2P$ is $K$-rational. On the other hand, since $\sigma$ maps $P$ to a point of order $20,$ we get
$$2P=\sigma(2P) = 2aP+ 2bQ'$$
and hence $a=1,11$ and $b=0,2.$ If $b=0,$ then $E$ has a $K$-rational cyclic $20$-isogeny generated by $P.$ It follows from applying \cite[Lemma 7]{Naj16} to the independent isogenies $\langle 2Q' \rangle$ and $\langle P \rangle$ of $E$ that $E$ is isogenous to an elliptic curve $E'/K$ with a cyclic $40$-isogeny. But it contradicts Theorem \ref{isogeny}. 
If $b=2,$ then $\sigma(P)=P+2Q'$ or $11P+2Q'.$ We consider the quotient curve $\overline{E}:= E/\langle 2Q' \rangle $ resp. $E/ \langle 10P {+} 2Q' \rangle $ with the isogeny $\phi: E \rightarrow \overline{E}.$ Writing $\overline{P}$ for the image of $P$ under $\phi,$ we have  $\sigma(\overline{P})= \overline{P}.$ So $\overline{E}$ has a $K$-rational $20$-torsion point $\overline{P},$ a contradiction by Theorem \ref{mainlist}.
\vspace{.2cm}

$\textit{(3)}$ If $D=-7$ and $E(K)_{\text{tor}} \simeq \mathbb{Z}/2\mathbb{Z} \oplus \mathbb{Z}/8\mathbb{Z},$ then by Theorem \ref{mainthm1} and Proposition \ref{mainprop1},  $E(L)_{\text{tor}} \simeq \mathbb{Z}/m\mathbb{Z} \oplus \mathbb{Z}/n\mathbb{Z}$ where
	\begin{equation} \label{stuff2}
 m \in \{2,4\}, \  n \in \{8,16,32\} \ \ \text{or} \ \
	 m=8,\  n \in \{8,16\}.
	\end{equation}
 
	By Table \ref{table1}, the groups $\mathbb{Z}/4\mathbb{Z} \oplus \mathbb{Z}/8\mathbb{Z}, \mathbb{Z}/2\mathbb{Z} \oplus \mathbb{Z}/{16}\mathbb{Z}$ are realized as $E(L)_{\text{tor}}.$ Moreover, \cite[Theorem 8]{BN16} excludes the subgroups $\mathbb{Z}/8\mathbb{Z} \oplus \mathbb{Z}/8\mathbb{Z}, \mathbb{Z}/4\mathbb{Z} \oplus \mathbb{Z}/{16}\mathbb{Z},$ which further rules out 
	$\mathbb{Z}/8\mathbb{Z} \oplus \mathbb{Z}/16\mathbb{Z},$ 
	$\mathbb{Z}/4\mathbb{Z} \oplus \mathbb{Z}/{32}\mathbb{Z}.$ It remains to disprove the occurrence of $\mathbb{Z}/2\mathbb{Z} \oplus \mathbb{Z}/32\mathbb{Z}.$ For this purpose, we will prove that $E(L)[32] \not\simeq \mathbb{Z}/2\mathbb{Z} \oplus \mathbb{Z}/32\mathbb{Z}.$ By way of contradiction, we assume $E(L)[32]\simeq \mathbb{Z}/2\mathbb{Z} \oplus \mathbb{Z}/{32}\mathbb{Z}.$  Fix a point $Q$ of order 2 and a point $P$ of order 32 such that $E(L)[32]= \langle Q, P \rangle.$ Then $E(K)_{\text{tor}} = \langle Q, 4P \rangle.$ The non-trivial element $\sigma $ of Gal($L/K)$ fixes $Q$ and sends $P$ to  a point of order 32, which must be of the form $aP$ or $Q + aP$ where $a$ is odd. Applying $\sigma$ to itself once again, we see that
	$$ P = \sigma^2(P) = a^2P $$
	which implies $a^2 \equiv 1 (\text{mod}\ 32)$ so $a$ can only be 1, 15, 17 or 31.  For all eight possibilities of $\sigma(P),$ $\sigma$ maps $2P$ to either $2P$ or $-2P.$ If $\sigma(2P) = 2P$ then the point $2P$ of order 16 would be $K$-rational, contradicting the assumption. Also $\sigma(2P)=-2P$ cannot happen, otherwise $4P$ would be mapped to its inverse which gives a contradiction since $4P$ is $K$-rational.

Let $D\neq -7$ and $E(K)_{\text{tor}} \simeq \mathbb{Z}/2\mathbb{Z} \oplus \mathbb{Z}/8\mathbb{Z}.$ It follows from Proposition \ref{mainprop1}, Theorem \ref{mainthm1}, and \cite[Theorem 8]{BN16} that the non-trivial possibilities for $E(L)_{\text{tor}}$ are $\mathbb{Z}/4\mathbb{Z} \oplus \mathbb{Z}/8\mathbb{Z}$ and $ \mathbb{Z}/2\mathbb{Z} \oplus \mathbb{Z}/{16}\mathbb{Z}.$ If $E(L) \simeq \mathbb{Z}/4\mathbb{Z} \oplus \mathbb{Z}/8\mathbb{Z}$, then in particular $\mathbb{Z}/{4}\mathbb{Z} \oplus \mathbb{Z}/4\mathbb{Z} \subseteq E(L)$ and so by the Weil pairing $L=K(\sqrt{-1}).$ However, the modular curve $X_1(4,8)$ is isomorphic (over $\Q(\sqrt{-1})$) to the elliptic curve with the Cremona label 32a2 \cite[Lemma 13]{Najder} and 
\begin{align}\label{4x8}
	\begin{cases}
 \text{if}\ \ D\neq -2, \  \text{then}\ X_1(4,8)(L) \simeq C_2 \oplus C_4\ \\
\text{if}\ \ D=-2, \ \text{then} \  X_1(4,8)(L)  \simeq C_4 \oplus C_4
\end{cases}
 \end{align}
where all the points are cusps. Therefore, we get a contradiction. It remains to show that $E(L)_{\text{tor}} \not\simeq \mathbb{Z}/2 \mathbb{Z} \oplus \mathbb{Z}/{16}\mathbb{Z}.$ By way of contradiction, let $E(L)_{\text{tor}} \simeq \mathbb{Z}/2\mathbb{Z} \oplus\mathbb{Z}/{16}\mathbb{Z}.$ Fix a point $Q$ of order $2$ and a point $P$ of order $16$ such that $E(L)_{\text{tor}}= \langle Q,P \rangle.$ Then every point $aP+bQ$ with odd $a$ has order 16, and hence lies outside $E(K)_{\text{tor}}.$ So the 16 points $aP+bQ$ with even $a$ must all lie in $E(K)_{\text{tor}}$ and hence fixed by $\sigma.$ In particular, this holds for $2P.$ From $\sigma(P)=aP+bQ$ with odd $a,$ we get 
$$ 2P=\sigma(2P)=2aP+2bQ. $$
So $a=1$ or $9.$  If $b=0,$ then $E$ has a $K$-rational cyclic $16$-isogeny generated by $P.$ Together with the independent $K$-rational $2$-torsion point $Q,$ we can apply \cite[Lemma 7]{Naj16} from which one obtains an elliptic curve over $K$ with a cyclic $32$-isogeny. However, Magma computes that the modular curve $X_0(32)$ has no non-cuspidal $K$-rational points for $D\neq -7.$ Hence, we get a contradiction. If $b=1,$ then $\sigma(P)=P+Q$ or $9P+Q.$ If we consider the elliptic curve $\overline{E}:=E/\langle Q \rangle $ resp. $E/\langle 8P+Q\rangle$, with the isogeny $\phi: E \rightarrow \overline{E},$ then $\sigma(\overline{P})=\overline{P}$ where $\overline{P}$ is the image of $P$ under $\phi.$ Hence, $\overline{E}$ has a $K$-rational $16$-torsion point $\overline{P}$, a contradiction by Theorem \ref{mainlist}.
\vspace{.2cm}

$\textit{(4)}$\	Suppose  $E(K)_{\text{tor}} \simeq \mathbb{Z}/2\mathbb{Z} \oplus \mathbb{Z}/6\mathbb{Z}.$ By Theorem \ref{mainthm1} and Proposition \ref{mainprop1}, we know
\begin{itemize}	
\item	if $D=-2$ then the torsion groups which can appear as $E(L)_{\text{tor}}$ are 
	\begin{align}\label{11}
	\mathbb{Z}/{2}\mathbb{Z} \oplus \mathbb{Z}/6\mathbb{Z},\ \mathbb{Z}/2\mathbb{Z} \oplus \mathbb{Z}/{12}\mathbb{Z},\ \mathbb{Z}/4\mathbb{Z} \oplus \mathbb{Z}/{12}\mathbb{Z}.
	\end{align}
	
\item	if $D=-7,-11$, $E(L)_{\text{tor}}$ is isomorphic to either one of the groups in \eqref{11} or 
	\begin{align}\label{12}
	\mathbb{Z}/2\mathbb{Z} \oplus  \mathbb{Z}/{18}\mathbb{Z},\
	&\mathbb{Z}/2\mathbb{Z} \oplus \mathbb{Z}/{36}\mathbb{Z},\
	\mathbb{Z}/4\mathbb{Z} \oplus \mathbb{Z}/{36}\mathbb{Z},\ \mathbb{Z}/6\mathbb{Z} \oplus \mathbb{Z}/6\mathbb{Z},\\
	\nonumber 
	&\mathbb{Z}/6\mathbb{Z} \oplus \mathbb{Z}/{12}\mathbb{Z},\ \mathbb{Z}/12\mathbb{Z} \oplus \mathbb{Z}/12\mathbb{Z}.
	\end{align}
	
\item	 if $D=-19,-43,-67,-163,$ then  $E(L)_{\text{tor}}$ is isomorphic to either one of the groups in \eqref{11},\eqref{12} or  $\mathbb{Z}/2\mathbb{Z} \oplus \mathbb{Z}/{24}\mathbb{Z},\ \mathbb{Z}/4\mathbb{Z} \oplus \mathbb{Z}/{24}\mathbb{Z}.$
\end{itemize}
	 
	 The groups $\mathbb{Z}/3\mathbb{Z} \oplus \mathbb{Z}/{12}\mathbb{Z}, \mathbb{Z}/4\mathbb{Z} \oplus \mathbb{Z}/{12}\mathbb{Z}$ are ruled out as subgroups by \cite[Theorem 8]{BN16} and so are $\mathbb{Z}/6\mathbb{Z} \oplus \mathbb{Z}/{12}\mathbb{Z},$ $\mathbb{Z}/4\mathbb{Z} \oplus \mathbb{Z}/{24}\mathbb{Z},\mathbb{Z}/4\mathbb{Z} \oplus \mathbb{Z}/{36}\mathbb{Z}$ and $\mathbb{Z}/12\mathbb{Z} \oplus \mathbb{Z}/12\mathbb{Z}.$ Moreover, Table \ref{tab2} shows an example where $E(L)_{\text{tor}} \simeq \mathbb{Z}/2\mathbb{Z} \oplus \mathbb{Z}/{12}\mathbb{Z}$ which establishes part $(i)$ of the assertion. If $E(L)[9]=\mathbb{Z}/9\mathbb{Z},$ then by Lemma \ref{lem2} we have
	$$\mathbb{Z}/9\mathbb{Z} \simeq E(K)[9] \oplus E^{d}(K)[9] \simeq \mathbb{Z}/3\mathbb{Z} \oplus E^{d}(K)[9] $$
	where $L=K(\sqrt{d})$ for some $d \in K.$ But the left hand side of the isomorphism is a cyclic group whereas the right hand side is not, a contradiction. As a result, we narrow down the list \eqref{12} by eliminating the groups $\mathbb{Z}/2\mathbb{Z} \oplus \mathbb{Z}/{18}\mathbb{Z}$ and $ \mathbb{Z}/2\mathbb{Z} \oplus \mathbb{Z}/{36}\mathbb{Z}.$ Notice that $\mathbb{Z}/6\mathbb{Z} \oplus \mathbb{Z}/6\mathbb{Z}$ is realized by Table \ref{tab4}.

	It remains to prove that $E(L)_{\text{tor}} \not\simeq \mathbb{Z}/2\mathbb{Z} \oplus \mathbb{Z}/{24}\mathbb{Z}$ for any $K$ in $\mathcal{E} \subseteq \mathcal{S}$ where $\mathcal{E}=\{\mathbb{Q}(\sqrt{D}) : D=-19,-43,-67,-163\} .$  Given $K$ in $\mathcal{E}$, let $F$ be as in \eqref{ffff}, the maximal elementary abelian $2$-extension of $K.$ For this purpose, we will show that $E(F)$ does not have a point of order $8.$ Let $E$ be given by $y^2=x(x+\alpha)(x+\beta)$ with $\alpha,\beta \in \mathcal{O}_K.$ If $E(F)$ has a point of order 8, then without loss of generality, there exists $s,t,z \in \mathcal{O}_F$ such that
	$\alpha = s^4,\   \beta=t^4$  with   $s^2+t^2=z^2.$
	It follows that either $\sqrt{\alpha}$ or $-\sqrt{\alpha} $ is a square in $F.$ Applying Lemma \ref{lem7}, we have either $\alpha$ or $-\alpha$ is a square in $K.$ As the same discussion applies to $\beta$, there are four possibilities:
	\begin{align*}
		&(i)\ \alpha=s_0^2 \ \text{and}\ \beta=t_0^2, \ \ \ \
		(iii)\ \alpha=s_0^2 \ \text{and} \ \beta=-t_0^2,\\ 
		&(ii)\ \alpha=-s_0^2\  \text{and}\ \beta=t_0^2, \  
		(iv)\ \alpha=-s_0^2 \ \text{and} \ \beta=-t_0^2
	\end{align*}
	for some non-zero $s_0,t_0 \in K.$
	\vspace{.1cm}
	
	Case $(i):$ Since $\mathcal{O}_K$ is integrally closed, if $a \in \mathcal{O}_K$ is a square in $K$ then it is a square in $\mathcal{O}_K.$ Thus, $\alpha$ and $\beta$ are both squares in $\mathcal{O}_K.$ It follows that $E(K)$ has a point of order 4 by Theorem \ref{newman}, which is a contradiction by assumption. 
	\vspace{.1cm}
	
	Case $(ii):$ As $E(K)$ has a point of order $3$, there are non-zero $a,b,c \in \mathcal{O}_K$ such that $\frac{a}{b} \notin \{-2,-1,-\frac{1}{2},0,1\}$ and
	$$ \alpha =a^3(a+2b)c^2 \ \ \ \  \beta=b^3(b+2a)c^2$$
	Together with the assumption, there exists a solution to the following system of equations
	$$ -s_0^2=a^3(a+2b)c^2 $$
	$$ t_0^2=b^3(b+2a)c^2 $$
	Since $b,c,b+2a \neq 0,$ we have
	$$
	\left(\frac{s_0}{t_0}\right)^2 = -\left(\frac{a}{b}\right)^3 \frac{a+2b}{b+2a}.                     
	$$
	Letting $y=\frac{s_0}{t_0}$ and $x=\frac{a}{b}$ yields that
	
	\begin{align}\label{tenten}
	y^2 = -x^3\frac{x+2}{2x+1} 
	\end{align}
	since $x \neq \frac{1}{2}.$ Let $C$ be the curve defined by  \eqref{tenten}  and $\overline{C}$ its projective closure. By Magma, there exists a birational map $ \phi:  \overline{C}  \rightarrow E'$  (over $\mathbb{Q})$ such that $E'$ is an elliptic curve defined by 
	$ y^2=x^3-5x^2+4x$
	and 
           $$ \phi([x, y, z])  =  [2x^3 + 4x^2z + 4y^2z, -2xyz - 4yz^2, y^2z].$$
Magma computes
	$E'(K) \simeq \mathbb{Z}/2\mathbb{Z} \oplus \mathbb{Z}/2\mathbb{Z}$ 
	for each $K$ from $\mathcal{E}.$ Moreover, neither the points of $\phi^{-1}(E'(K))$ nor the set of non-regular points of $\phi$ give rise to a point $(x,y)$ satisfying the hypothesis where $x \notin \{-2,-1,-\frac{1}{2},0,1\}$ and $ y\neq 0.$ Therefore, we get a contradiction.
	\vspace{.1cm}
	
	Case $(iii):$ It follows similarly as in the previous case.
	\vspace{.1cm}
	
	Case $(iv):$ In this case we have
	\begin{align}\label{irmak}
		s^4=\alpha=-s_0^2 \ \ \ \  t^4=\beta=-t_0^2,
	\end{align}
	and so $s^2=\pm is_0.$ Note that $i \in F.$  The curve defined by $s^2+t^2=z^2$ is parametrized by $(1-m^2,2m,1+m^2)$ over $F.$  By replacing $t$ with $2m$ in \eqref{irmak}  and $s$ with $1-m^2$ in $s^2=\pm is_0$ we have $4m^4=-\frac{t_0^2}{4} \in K$ and 
	$$ 4m^4 = (-is_0 + m^4+1)^2 \ \text{or} \ 4m^4=(is_0+m^4+1)^2 $$
	We can re-write the above as
	$$
	-\frac{t_0^2}{4} = \left(\pm is_0-\frac{t_0^2}{16} + 1 \right)^2.
	$$
	
	One can see that the right side is in $K$ if $t_0=\pm 4.$ But then $s_0=\pm 2,$ which leads to the elliptic curve $E'$ defined above. However, $E'(K)_{\text{tor}}$ is not isomorphic to  $\mathbb{Z}/2\mathbb{Z} \oplus \mathbb{Z}/6\mathbb{Z}$ which gives a contradiction by assumption.
\vspace{.2cm}

$\textit{(5)}$\	Suppose $E(K)_{\text{tor}} \simeq \mathbb{Z}/2\mathbb{Z} \oplus \mathbb{Z}/4\mathbb{Z}$ and $E[4] \subseteq E(L).$ In particular $L=K(i)$ by the Weil pairing. Note that $\mathbb{Z}/4\mathbb{Z} \oplus \mathbb{Z}/4\mathbb{Z}$ is realized for each $K$ in $\mathcal{S}$ by Table \ref{tab5}. It follows from Theorem \ref{mainthm1} and Proposition \ref{mainprop1} using \cite[Theorem 8]{BN16} that the remaining possibility for $E(L)_{\text{tor}}$ is $\mathbb{Z}/4\mathbb{Z} \oplus \mathbb{Z}/8\mathbb{Z}.$ If $D=-7,$ then it is realized by Table \ref{table1}. If $D\neq -7$ and $E(L)_{\text{tor}} \simeq \mathbb{Z}/4\mathbb{Z} \oplus \mathbb{Z}/8\mathbb{Z},$ then the original curve $E$ would be induced by a non-cuspidal $L$-rational point of the modular curve $X_1(4,8)$, a contradiction by \eqref{4x8}.

    Suppose $E[4] \not\subseteq E(L).$ By Theorem \ref{mainthm1} and Proposition \ref{mainprop1} using \cite[Theorem 8]{BN16}, we have 
    \begin{itemize}
   \item if $D = -2,-11,$ then the possible groups for $E(L)_{\text{tor}}$ are 
   \begin{align}\label{not4x4}
   \mathbb{Z}/2\mathbb{Z} \oplus \mathbb{Z}/4\mathbb{Z},\ \mathbb{Z}/2\mathbb{Z} \oplus \mathbb{Z}/8\mathbb{Z}, \  \mathbb{Z}/{2}\mathbb{Z} \oplus \mathbb{Z}/{16}\mathbb{Z}.
   \end{align}
\item if $D=-7,$ then $E(L)_{\text{tor}}$ is isomorphic to either one of the groups in \eqref{not4x4} or
 $$\mathbb{Z}/2\mathbb{Z} \oplus \mathbb{Z}/32\mathbb{Z}.$$
 \item if $D=-19,-43,-67,-163,$ then $E(L)_{\text{tor}}$ is isomorphic to either one of the groups in \eqref{not4x4} or $$\mathbb{Z}/2\mathbb{Z} \oplus \mathbb{Z}/{12}\mathbb{Z},\ \mathbb{Z}/2\mathbb{Z} \oplus \mathbb{Z}/{24}\mathbb{Z}.$$
 \end{itemize}

   If $E(L)[16] \simeq \mathbb{Z}/2\mathbb{Z} \oplus \mathbb{Z}/{16}\mathbb{Z},$ then $E(L)[16] =\langle Q, P \rangle $ where $Q$ is a 2-torsion point and $P$ is a $16$-torsion point, So, $E(K)_{\text{tor}} \simeq \langle Q, 4P \rangle.$ Then we have $\sigma(P)=aP$ or $aP+Q$ where $a$ is odd. Applying $\sigma$ twice implies $ P=a^2P $ and so  $a^2\equiv 1$ mod ($16$). Then $a$ must be $1,7,9$ or $15.$ For all eight possibilities of $\sigma(P),$ we have either $\sigma(2P)=2P$ or $\sigma(2P)=-2P.$ The former is not possible by assumption. The later implies that $\sigma(4P)=-4P,$ which yields a contradiction as $4P$ is $K$-rational. This rules out the groups $\mathbb{Z}/{2}\mathbb{Z} \oplus \mathbb{Z}/{16}\mathbb{Z}$ and $\mathbb{Z}/{2}\mathbb{Z} \oplus \mathbb{Z}/{32}\mathbb{Z}.$ One can see Table \ref{tab3} for the occurrence of $\mathbb{Z}/2\mathbb{Z} \oplus \mathbb{Z}/12\mathbb{Z}.$ Lastly, replacing $E$ by $E^{d}$ if necessary, the group $\mathbb{Z}/2\mathbb{Z} \oplus \mathbb{Z}/{24}\mathbb{Z}$ can be excluded as in the proof of $\textit{(4)}$ above. 
\vspace{.2cm}

$\textit{(6)}$\ Suppose $E(K)_{\text{tor}} \simeq \mathbb{Z}/2\mathbb{Z} \oplus \mathbb{Z}/2\mathbb{Z}.$ As $E(K(\sqrt{d})) \simeq E^d(K(\sqrt{d})),$ it suffices to study the non-trivial  growth of $E^d(K)_{\text{tor}}$ over $K(\sqrt{d})$. If $E^d(K)_{\text{tor}} \not\simeq \mathbb{Z}/2\mathbb{Z} \oplus \mathbb{Z}/2\mathbb{Z},$ then we apply a previously proven case (see Table \ref{tab2}, \ref{table1} for examples).
If $E^d(K)_{\text{tor}} \simeq \mathbb{Z}/2\mathbb{Z} \oplus \mathbb{Z}/2\mathbb{Z},$ then we have from Proposition \ref{mainprop1} that $E(L)_{\text{tor}}$ is isomorphic to $\mathbb{Z}/2\mathbb{Z} \oplus \mathbb{Z}/2\mathbb{Z},$ $\mathbb{Z}/2\mathbb{Z} \oplus \mathbb{Z}/{4}\mathbb{Z},$  or $\mathbb{Z}/4\mathbb{Z} \oplus \mathbb{Z}/4\mathbb{Z}$ (see Table \ref{tab88}).

\bibliographystyle{plain}
\bibliography{bibfile}

\end{document}